\documentclass[11pt]{amsart}

\usepackage{amssymb,amsmath}
\usepackage{bbm}
\usepackage{a4wide}
\usepackage{graphicx}

\newtheorem{theorem}{Theorem}
\newtheorem{prop}{Proposition}
\newtheorem{lemma}{Lemma}
\newtheorem{coro}{Corollary}
\newtheorem{fact}{Fact}
 
{
\theoremstyle{definition}

\newtheorem{remark}{Remark}
\newtheorem{example}{Example}
}

\newcommand{\ts}{\hspace{0.5pt}}
\newcommand{\nts}{\hspace{-0.5pt}}

\newcommand{\RR}{\mathbb{R}\ts}

\newcommand{\ZZ}{\mathbb{Z}}
\newcommand{\TT}{\mathbb{T}}
\newcommand{\NN}{\mathbb{N}}
\newcommand{\FF}{\mathbb{F}}
\newcommand{\rA}{\nts\mathrm{A}\nts}
\newcommand{\rF}{\ts \mathrm{F}}
\newcommand{\cS}{\mathcal{S}}
\newcommand{\cR}{\mathcal{R}}
\newcommand{\cG}{\mathcal{G}}

\newcommand{\one}{\mathbbm{1}}

\newcommand{\exend}{\hfill$\Diamond$}

\DeclareMathOperator{\lcm}{lcm\ts}
\DeclareMathOperator{\mgcd}{mgcd}
\DeclareMathOperator{\ord}{ord}

\DeclareMathOperator{\trace}{tr}
\DeclareMathOperator{\card}{card\ts}
\DeclareMathOperator{\GL}{GL}
\DeclareMathOperator{\PGL}{PGL}
\DeclareMathOperator{\SL}{SL}
\DeclareMathOperator{\diag}{diag}
\DeclareMathOperator{\fix}{Fix}
\DeclareMathOperator{\per}{\kappa}
\DeclareMathOperator{\Mper}{per}
\DeclareMathOperator{\Mat}{Mat}

\begin{document}

\title[]
{Orbit structure and (reversing) symmetries \\[1mm]
of toral endomorphisms on rational lattices}

\author{Michael Baake}

\author{Natascha Neum\"{a}rker}
\address{Fakult\"at f\"ur Mathematik, Universit\"at Bielefeld, \newline
\hspace*{\parindent}Postfach 100131, 33501 Bielefeld, Germany}
% \email{mbaake@math.uni-bielefeld.de}
% \urladdr{http://www.math.uni-bielefeld.de/baake}

\author{John A.~G.~Roberts}
\address{School of Mathematics and Statistics, 
University of New South Wales, \newline
\hspace*{\parindent}Sydney, NSW 2052, Australia}
% \email{jag.roberts@unsw.edu.au}
% \urladdr{http://www.maths.unsw.edu.au/\~{}jagr}

\begin{abstract} 
   We study various aspects of the dynamics induced by integer
   matrices on the invariant rational lattices of the torus in
   dimension $2$ and greater.  Firstly, we investigate the orbit
   structure when the toral endomorphism is not invertible on the
   lattice, characterising the {\em pretails} of eventually periodic
   orbits. Next we study the nature of the symmetries and reversing
   symmetries of toral automorphisms on a given lattice, which has
   particular relevance to (quantum) cat maps.
\end{abstract}

\maketitle

\section{Introduction}

Toral automorphisms or \emph{cat maps}, by which we mean the action of
matrices $M \in \GL(d,\ZZ)$ on the $d$-torus $\TT^{d}$, are a widely
used and versatile class of dynamical systems, see \cite{WalBook,KH}
for some classic results in the context of ergodic theory. Of
particular interest are the hyperbolic and quasi-hyperbolic ones,
which are characterised by having no root of unity among their
eigenvalues.  All periodic orbits of such automorphisms lie on the
rational (or finite) invariant lattices $L_{n} = \{ x \in \TT^{d} \mid
nx = 0 \bmod 1 \}$, which are also known as the $n$-division
points. One can encode the possible periods of a toral automorphism
$M$ on $\TT^{d}$ via the dynamical zeta function in a systematic way,
which is always a rational function \cite{BLP,DEI}. The literature on
classifying periodic orbits of toral automorphisms when $d=2$ is vast
(compare \cite{DF,G,PV} and references therein). An extension beyond
$d=2$ is difficult due to the fact that the conjugacy problem between
integer matrices is then much harder because (unlike $d=2$) no
complete set of conjugacy invariants mod $n$ is known. Therefore, our
focus will also be on $d=2$, with occasional extensions to higher
dimensions.

The larger ring $\Mat(d,\ZZ)$ of toral endomorphisms (which includes
integer matrices without integer inverses) has received far less
attention \cite{ATW, BV, BLP}, particularly those in the complement of
$\GL(d,\ZZ)$. Note that the resulting dynamics induced by $M \in
\Mat(d,\ZZ)\setminus \GL(d,\ZZ)$ on a finite lattice $L_{n}$ may or
may not be invertible.  In the latter case, beyond periodic orbits,
there exist \emph{eventually periodic} orbits which possess points
that lead into a periodic orbit.  We call these points and the
periodic point to which they attach the `pretails' to the periodic
orbit (see Eq.~\eqref{eq:pretail} for a formal definition). The action
of $M$ induces a directed graph on $L_n$ (e.g.\ see our three figures
below).  Alternatively, the pretails can be combined to form a rooted
\emph{tree} which is a characteristic attribute to any pair $(M,
L_{n})$.

As well as their interest from a mathematical viewpoint, toral
automorphisms also have been well-studied from a physics perspective,
in particular as quantum cat maps (see \cite{KM,DW,KR} and references
therein). Here, the action of the integer matrix on a rational lattice
$L_{n}$, for some $n$, is all-important as quantum cat maps and their
perturbations are built from (classical) cat maps and their
perturbations \emph{restricted to a particular rational lattice}
(called the Wigner lattice in this instance). There has been recent
interest in dealing with so-called pseudo-symmetries of quantum cat
maps that are manifestations of local symmetries of cat maps
restricted to some rational lattice \cite{KM,DW,KR}. Although, in
the context of quantisation, matrices from the group $\mathrm{Sp}
(2d,\ZZ)$ play the key role, we prefer to work with the
larger group of unimodular integer matrices and consider the
former as a special case.

\medskip

The main aims of this paper are twofold:\ (i) to elucidate the orbit
structure of toral endomorphisms on rational lattices, equivalently
the periodic orbits together with the related pretail tree structure;
(ii) to further characterise the nature of symmetries or (time)
reversing symmetries of toral automorphisms, these being automorphisms
of the torus (or of a rational lattice) that commute with the cat map,
respectively conjugate it into its inverse. 

We expand a little on our results, where we refer to the actual
formulation below in the paper. The results are readable without 
the surrounding notational details.

With respect to aim (i), Section \ref{sec:orbit} characterises the
splitting of $L_{n}$ into periodic and eventually periodic points
under a toral endomorphism $M$. Every periodic point has a pretail
graph isomorphic to that of the fixed point $0$ (Corollary
\ref{coro:p-tree}), which is trivial if and only if $M$ is invertible
on $L_{n}$. In general, the pretail tree codes important information
on the action of $M$. One question in this context is whether all
\emph{maximal} pretails have the same length, for which we give a
partial answer via a sufficient condition on $\ker(M)$ in Proposition
\ref{prop:pre-length}. Given $M$, the lattice $L_{n}$ can be
decomposed into into 2 invariant submodules, one of which captures the
invertible part of $M$ and the other the nilpotent part. This way, we
are able to characterise the dynamics that is induced by $M$ on
$L_{n}$ in the case of $n=p^r$, $p$ prime, in Corollary
\ref{c:blockdiagonal} and Lemma \ref{l:mipoDecomp}.

Our contribution towards aim (ii) continues the investigations from
\cite{BRcat,BRtorus,BRW}. The key quantity for integer matrices of
dimension $2$ is the $\mgcd$ (see Eq.~\ref{eq:mgcd-def} below), and
one consequence of \cite[Thm.~2]{BRW} is that $M\in \SL (2,\ZZ)$ is
always conjugate to its inverse on $L_{n}$, for each $n\in\NN$. The
conjugating element -- called a reversing symmetry or reversor -- is
an integer matrix that has an integer matrix inverse on $L_{n}$, which
typically depends on $n$. In this way, any $\SL (2,\ZZ)$ matrix that
fails to be conjugate to its inverse on the torus (e.g.\
$M=\left( \begin{smallmatrix} 4 & 9 \\ 7 & 16 \end{smallmatrix}
\right)$ from \cite[Ex.~2]{BRcat}) is still conjugate to its inverse
on every rational lattice.  In \cite{BRW}, we did not consider the
nature of the reversor on the lattice. Theorem \ref{thm:local-rev} of
Section \ref{sec:revsym} establishes that it is an
orientation-reversing involution, what is called an anticanonical
(time-reversal) symmetry in the language of \cite{KM}.  Section
\ref{sec:4.2} uses normal forms of $\GL(2,\FF_{p})$ to characterise
the symmetries and possible reversing symmetries of such matrices; the
underlying structure of the conjugacy classes of $\GL (2,\FF_{p})$ is
summarised in Table~\ref{tab:finite-field}.  The symmetry structure
has some extensions to higher dimensions (Section \ref{sec:4.3}) and
to general modulus $n$ (Section \ref{sec:4.4}). Section \ref{sec:4.5}
presents some results for the case when $M \in \GL(d,\FF_{p})$ has a
root in the same group.

The structure of the paper is as follows. Section \ref{sec:prelim}
summarises some properties of integer matrices that we use later in
the paper, with some reformulations or slight generalisations that we
find useful. In particular, throughout the paper, we formulate the
results for arbitrary dimension whenever it is possible without extra
complications, though this is not our main focus.  As described,
Section \ref{sec:orbit} addresses aim (i) above, while Section
\ref{sec:revsym} deals with aim (ii). In the Appendix, we briefly
discuss two classic examples of toral automorphisms for $d=2$ and some
aspects of their dynamics.

\section{Preliminaries and powers of integer matrices} 
\label{sec:prelim}

The purpose of this section is to summarise important properties of
and around integer matrices that are needed later on, with focus
on those that are not standard textbook material. At the same time, we
introduce our notation. For general background on integer
matrices and their connections to algebraic number theory, we
refer to the classic text by Taussky \cite{Tau}.

\subsection{Lattices, rings and groups}

The most important lattices on the torus $\TT^{d}=\RR^{d}/\ZZ^{d}$,
which is a compact Abelian group, consist of the $n$-division points
\begin{equation} \label{define-L}
   L_n \, := \,  \{ x\in\TT^{d} \mid nx=0 \; \mbox{\rm (mod $1$)}\}
   \, = \, \bigl\{ \big( \tfrac{k_{1}}{n},\ldots,
   \tfrac{k_{d}}{n}\big)
   \mid 0 \le k_{i} < n \mbox{ for all $1\le i\le d$} \bigr\}\ts , 
\end{equation}
with $n\in\NN$.  Clearly, the $L_{n}$ are invariant under toral
endomorphisms (with the action of the representing matrices taken mod
$1$). It is sometimes easier to replace $L_n$ by the set $\tilde{L}_n
:= \{ (k_{1},\ldots, k_{d})\mid 0 \le k_{i} < n \}$, with the
equivalent action of $M$ defined mod $n$.  This also applies to
various theoretical arguments involving modular arithmetic.
Consequently, we use $L_n$ (with action of $M$ mod $1$) and
$\tilde{L}_n$ (with action mod $n$) in parallel.

Our discussion will revolve around the residue class ring $\ZZ/n\ts
\ZZ$ with $n\in\NN$, which is a principal ideal ring, but not a
domain, unless $n=p$ is a prime. In the latter case, $\ZZ/p\ts
\ZZ=\FF_{p}$ is the finite field with $p$ elements, while the ring has
zero divisors otherwise.  For general $n$, the unit group
\[
   (\ZZ/n\ts \ZZ)^{\times} \, = \,
   \{ 1\le m \le n \mid \gcd (m,n) = 1 \}
\]
is an Abelian group (under multiplication) of order
$\phi (n)$, where $\phi$ is Euler's totient function from elementary
number theory \cite{Hasse}. In general, it is not a cyclic group.

The integer matrices mod $n$ form the finite ring $\Mat(d,\ZZ/n\ts\ZZ)$
of order $n^{d^2}$. The invertible elements in it form the group
$\GL (d, \ZZ/n \ts \ZZ) = \{ M\in\Mat(d,\ZZ/n\ts\ZZ)\mid
\det (M) \in (\ZZ/n\ts \ZZ)^{\times} \}$. If $n=p_{1}^{r^{}_{1}}\cdots
p_{\ell}^{r^{}_{\ell}}$ is the standard prime decomposition, one
finds
\begin{equation} \label{eq:GL-order}
  \big\lvert \GL (d, \ZZ/n\ts \ZZ) \big\rvert \, = \,
  n^{d^{2}} \prod_{j=1}^{\ell} \frac
  {\big\lvert \GL (d,\FF_{p^{}_{j}}) \big\rvert}
  {p_{j}^{d^{2}}} \, ,
\end{equation}
where
\begin{equation}\label{eq:GL-Fp}
   \big\lvert \GL (d,\FF_{p}) \big\rvert \, = \,
   (p^{d} - 1) (p^{d} - p) \cdot \ldots \cdot
   (p^{d} - p^{d-1})
\end{equation}
is well-known from the standard literature \cite{Lang,LN}. Formula
\eqref{eq:GL-order} follows from the corresponding one for $n=p^r$ via
the Chinese remainder theorem, while the simpler prime power case is a
consequence of the observation that each element of a non-singular
matrix $M$ over $\ZZ/p^s\ts\ZZ$ can be covered (independently of all
other matrix elements) by $p$ elements in $\ZZ/p^{s+1}\ts \ZZ$ without
affecting its non-singularity.

Let us finally mention that $\SL (n, \ZZ/n\ts\ZZ)$, the
subgroup of matrices with determinant $1$, is a normal 
subgroup (it is the kernel of $\det\!:\, \GL (n, \ZZ/n\ts\ZZ)
\longrightarrow (\ZZ/n\ts\ZZ)^{\times}$). The factor group is
\[
   \GL (n, \ZZ/n\ts\ZZ)/\SL (n, \ZZ/n\ts\ZZ)
   \, \simeq \, (\ZZ/n\ts\ZZ)^{\times}
\]
and thus has order $\phi(n)$.

\subsection{Orbit counts and generating functions}

The orbit statistics of the action of a matrix $M\in\Mat (d,\ZZ)$ on
the lattice $L_n$ is encapsulated in the polynomial
\begin{equation} \label{eq:def-poly}
   Z_{n} (t) \, \, = \, \prod_{m\in\NN}
    (1-t^{m})^{c_{m}^{(n)}} , 
\end{equation}
where $c_{m}^{(n)}$ denotes the number of $m$-cycles of $M$ on $L_n$.
Recall that, if $a_m$ and $c_m$ denote the fixed point and orbit count
numbers of $M$ (dropping the upper index for a moment), they are
related by
\begin{equation} \label{a-from-c}
   a_m \, = \, \sum_{d\ts | m} d\, c_d 
   \quad \text{and} \quad
    c_m \, = \, \frac{1}{m} \sum_{d\ts |  m} 
   \mu \big( \tfrac{m}{d}\big) \ts  a_d\ts ,
\end{equation}
where $\mu(k)$ is the M\"obius function from elementary number theory
\cite{Hasse}.

Despite the way it is written, $Z_n$ is a \emph{finite} product and
defines a polynomial of degree at most $n^{d}$.  Note that the degree
of $Z_n$ can be smaller than $n^{d}$ (as the matrix $M$ need not be
invertible on $L_n$), but $Z_{n} (t)$ is always divisible by $(1-t)$,
because $0$ is a fixed point of every $M$. The polynomials $Z_{n}$ are
closely related \cite{BRW,Diss} to the zeta function of toral
endomorphisms, which can be calculated systematically; compare
\cite{BLP} and references therein. Dynamical zeta functions give
access to the distribution and various asymptotic properties of
periodic orbits \cite{DEI,Ruelle}, and also relate to topological
questions; compare \cite{Fel} for a systematic exposition of the
latter aspect in a more general setting. Further aspects on the
asymptotic distribution of orbit lengths on prime lattices can
be found in \cite{Keat}.

\subsection{Matrix order on lattices and plateau phenomenon}
Assume that $M$ is invertible on $L_n$ (hence also on
$\tilde{L}_{n}$). Then, its order is given by
\begin{equation} \label{per-one}
  \ord\ts (M,n) \, := \;
  \gcd\ts \{ m\in\NN_{0} \mid M^m \equiv\one \; \bmod{n} \}\ts .
\end{equation}
Clearly, $\ord\ts (M,1)=1$ in this setting. When $M$ is not invertible
on $L_n$, the definition results in $\ord\ts (M,n)=0$; otherwise,
$\ord\ts (M,n)$ is the smallest $m\in\NN$ with $M^m=\one$ mod $n$.

Let $M\in\GL(d,\ZZ)$ be arbitrary, but fixed. To determine $\ord(M,n)$
for all $n\ge 2$, it suffices to do so for $n$ an arbitrary prime
power, since the Chinese remainder theorem \cite{Hasse} gives
\begin{equation} \label{eq:power-decomp}
   \ord (M,n) \, = \, \lcm \bigl( 
    \mbox{$\ord (M,p_{1}^{r^{}_{\nts 1}}),
    \ldots, \ord (M,p_{\ell}^{r^{}_{\! \ell}})$} \bigr)
\end{equation}
when $n=p_{1}^{r^{}_{\nts 1}}\cdots p_{\ell}^{r^{}_{\!
\ell}}$ is the prime decomposition of $n$. It is clear that
$\ord (M, p^r) | \ord (M, p^{r+1})$ for all $r\in\NN$,
see also \cite[Lemma~5.2]{BV}.

Let us now assume that $M \in \Mat(d,\ZZ)$ is \emph{not} of finite
order, meaning that $M^{k} \ne \one$ for all $k\in\NN$, which excludes
the finite order elements of $\GL(d,\ZZ)$.  If $p$ is a prime, we then
obtain the unique representation
\begin{equation} \label{eq:power-rep}
   M^{\ord (M,p)} \, = \, \one + p^{s} B
\end{equation}
with $s\in\NN$ and an integer matrix $B \not\equiv 0$ mod $p$.
Starting from this representation, an application of the
binomial theorem for powers of it, in conjunction with the
properties of the binomial coefficients mod $p$, gives the 
following well-known result.

\begin{prop} \label{prop:powers-modp}
  Let\/ $M\in\Mat (d,\ZZ)$ be a matrix that is not of finite order.
  Fix a prime $p$ that does not divide $\det (M)$, and let $s$ be
  defined as in Eq.~{\nts}\eqref{eq:power-rep}.  

  When $p$ is odd or when $s\ge 2$, one has\/ $\ord (M, p^{i}) = \ord
  (M, p)$ for $1\le i \le s$, together with\/ $\ord (M, p^{s+i}) =
  p^{i} \ts \ord (M, p^{s})$ for all\/ $i\in\NN$.

  In the remaining case, $p=2$ and $s=1$, one either
  has\/ $\ord (M, 2^{r}) = 2^{r-1} \ts \ord (M,2)$ for
  all $r\in\NN$, or there is an integer\/ $t\ge 2$ so
  that\/ $\ord (M, 2^{i}) = 2\ts \ord (M, 2)$ for\/
  $2\le i \le t$ together with\/ $\ord (M, 2^{t+i})
  = 2^{i} \ts \ord (M, 4)$ for all\/ $i\in\NN$.
  \qed
\end{prop}

In what follows, we will refer to the structure described in 
Proposition~\ref{prop:powers-modp} as the \emph{plateau phenomenon}.
Such a plateau can be absent ($p$ odd with $s=1$, or the first case
for $p=2$), it can be at the beginning ($p$ odd with $s\ge 2$),
or it can occur after one step ($p=2$ when $t\ge 2$ exists as
described), but it cannot occur later on.

Proposition~\ref{prop:powers-modp} is a reformulation of
\cite[Thms.~5.3 and 5.4]{BV}, which are originally stated for
$M\in\GL(2,\ZZ)$. As one can easily check, the proofs do not depend on
the dimension. Similar versions or special cases were also given in
\cite{BF} and \cite{S} (with focus on $\SL(2,\ZZ)$-matrices), in
\cite{PV} (for the order of algebraic integers), in \cite{W} (for the
Fibonacci sequence), in \cite{BuS} (for linear quadratic recursions)
and in \cite{E} and \cite{Wa} (for general linear recursions).
Let us also mention that, based on the generalised Riemann
hypothesis, Kurlberg has determined a lower bound on the order
of unimodular matrices mod $N$ for a density $1$ subset
of integers $N$ in \cite{Kurl}.

\subsection{Powers of integer matrices}\label{sec:pim}

Consider a matrix $M\in\Mat(d,\ZZ)$ with $d\ge 2$ and characteristic
polynomial $P^{}_{M} (z) = \det (z\one - M)$, which (following
\cite{Wa}) we write as
\[
   P^{}_{M} (z) \, = \, z^{d} - c^{}_{1} z^{d-1}
   - c^{}_{2} z^{d-2} - \ldots - c^{}_{d-1} z 
   - c^{}_{d} \ts ,
\]
so that $c^{}_{d} = (-1)^{d+1}\det (M)$. Let us define a recursion by
$u^{}_{0} = u^{}_{1} = \ldots = u^{}_{d-2}=0$ and $u^{}_{d-1} = 1$
together with
\begin{equation} \label{eq:rec-def} 
   u^{}_{m} \, = \, \sum_{i=1}^{d} c^{}_{i} \ts
   u^{}_{m-i} \, = \,
   c^{}_{1} \ts u^{}_{m-1} + c^{}_{2} \ts u^{}_{m-2}
   + \ldots + c^{}_{d} \ts u^{}_{m-d}
\end{equation}
for $m\ge d$. This results in an integer sequence. Moreover,
when $c^{}_{d} \ne 0$, we also define
\[
   u^{}_{m} \, = \, c^{-1}_{d} \ts ( u^{}_{m+d} -
   c^{}_{1} \ts u^{}_{m+d-1} - \ldots - 
   c^{}_{d-1} \ts u^{}_{m+1} )
\]
for $m \le -1$. In particular, since $d\ge 2$, one always has
$u^{}_{-1} = 1/c^{}_{d}$ and $u^{}_{-2} = - c^{}_{d-1}/c^{2}_{d}$,
while the explicit form of $u_{m}$ with $m<-2$ depends on $d$.  Note
that the coefficients with negative index are rational numbers in
general, unless $\lvert c^{}_{d} \rvert = 1$.

The Cayley-Hamilton theorem together with \eqref{eq:rec-def} can be
used to write down an explicit expansion of powers of the matrix $M$
in terms of $M^k$ with $0\le k\le d-1$,
\begin{equation} \label{eq:matrix-power}
   M^{m} \, = \, \sum_{\ell=0}^{d-1} \gamma^{(m)}_{\ell} \ts M^{\ell} ,
\end{equation}
where the coefficients satisfy $\gamma^{(m)}_{\ell} = \delta^{}_{m,\ell}$
(for $0\le \ell,m \le d-1$) together with the recursion
\begin{equation}\label{eq:coeff-rec}
%   \gamma^{(n+1)}_{\ell} = \begin{cases}
%    c^{}_{d} \, \gamma^{(n)}_{d-1} , &  \text{if $\ell=0$} , \\[1mm]
%    c^{}_{d-\ell} \, \gamma^{(n)}_{d-1} + \gamma^{(n)}_{\ell-1} ,
%    & \text{if $1 \le \ell \le d-1$} ,   \end{cases}
    \gamma^{(n+1)}_{\ell} =
    c^{}_{d-\ell} \, \gamma^{(n)}_{d-1} + \gamma^{(n)}_{\ell-1} ,
\end{equation}
for $n \ge d-1$ and $0\le \ell \le d-1$, where $\gamma^{(n)}_{-1} :=
0$.  In particular, $\gamma^{(d)}_{\ell} = c^{}_{d-\ell}$. The
coefficients are explicitly given as
\begin{equation} \label{eq:power-coeff}
   \gamma^{(m)}_{\ell} \, = \,
   \sum_{i=0}^{\ell} c^{}_{d-i} \ts u^{}_{m-\ell-1+i}
   \, = \, u^{}_{m +d -\ell-1}\, - \! \sum_{i=1}^{d-\ell-1}
   c^{}_{d-\ell-i}\ts u^{}_{m -1+i} \ts ,
\end{equation}
where $m\ge d$ and the second expression follows from the first by
\eqref{eq:rec-def}.  Formulas \eqref{eq:matrix-power} and
\eqref{eq:power-coeff} can be proved by induction from $M^{d} =
c^{}_{1} M^{d-1} + c^{}_{2} M^{d-2} + \ldots c^{}_{d-1} M + c^{}_{d}
\one$. Eq.~\eqref{eq:matrix-power} holds for all $m\ge 0$
in this formulation.

When $\det (M) \ne 0$, the representation \eqref{eq:power-coeff} also
holds for $m<d$, as follows from checking the cases $0\le m < d$
together with a separate induction argument for $m<0$. In particular,
one then has
\[
\begin{split}
   M^{-1} = \; & c^{}_{d} \ts u^{}_{-2} \one +
   (c^{}_{d-1} \ts u^{}_{-2} + c^{}_{d} \ts u^{}_{-3}) M +
   (c^{}_{d-2} \ts u^{}_{-2} + c^{}_{d-1} \ts u^{}_{-3}
    + c^{}_{d} \ts u^{}_{-4}) M^{2} \\
   & + \ldots +
   (c^{}_{2} \ts u^{}_{-2} + c^{}_{3} \ts u^{}_{-3}
   + \ldots + c^{}_{d} \ts u^{}_{-d}) M^{d-2}
   + u^{}_{-1} M^{d-1} ,
\end{split}
\]
which is again an integer matrix when $\lvert c^{}_{d} \rvert = 1$.

\subsection{Results for $d=2$}
Let us look at matrices from $\Mat (2,\ZZ)$ more closely, and derive
one important result by elementary means. Consider
$M=\left(\begin{smallmatrix} a & b \\ c & d
  \end{smallmatrix} \right)$, set $D:= \det (M)$, $T:= \trace (M)$ and
define the matrix gcd (or mgcd for short) as
\begin{equation}\label{eq:mgcd-def}
   \mgcd (M) := \gcd (b,c,d-a) \ts ,
\end{equation}
which is another invariant under $\GL (2,\ZZ)$ conjugation.  Its
special role becomes clear from the following result, which is a
reformulation of \cite[Lemma~2 and Thm.~2]{BRW}. This will lead to
Corollary~\ref{coro:BRW} below.

\begin{lemma}\label{lem:mgcd}
  Two matrices\/ $M, M' \in \Mat (2, \ZZ)$ that are\/
  $\GL(2,\ZZ)$-conjugate possess the same mgcd,
  as defined in Eq.~\eqref{eq:mgcd-def}. More generally,
  the reductions modulo\/ $n$ of\/ $M$ and\/ $M'$ are\/ 
  $\GL (2,\ZZ/n\ZZ)$-conjugate for all\/ $n\ge 2$ if and
  only if the two matrices share the same trace, determinant
  and mgcd.  \qed
\end{lemma}

Returning to matrix powers, formula~\eqref{eq:matrix-power} simplifies
to
\begin{equation} \label{eq:iter2}
   M^{m} \, = \, u^{}_{m} M - D u^{}_{m-1} \one \ts , 
\end{equation}
where now $u^{}_{0} = 0$, $u^{}_{1} = 1$ and $u^{}_{m+1} = T u^{}_{m}
- D u^{}_{m-1}$ for $m\in\NN$; see \cite[Sec.~2.3]{BRW} for details.
Let $n\in\NN$ and assume $\gcd (n,D) = 1$.  This allows us to introduce
\[
   \per (n) \, := \, \mbox{period of } (u_{m})_{m\ge 0} \bmod{n}
\]
which is well-defined, as the sequence mod $n$ is then indeed periodic
without `pretail'. Recall that $(u_{m})_{m\ge 0}\bmod n$ must be
periodic from a certain index on, as a result of Dirichlet's pigeon
hole principle. Since $D$ is a unit in $\ZZ/n \ZZ$, the recursion
\eqref{eq:iter2} can be reversed, and $(u_{m})_{m\ge 0}\bmod n$ must
thus be periodic, with $\per (n)$ being the smallest positive integer
$k$ such that $u^{}_{k} = 0$ and $u^{}_{k+1}=1$ mod $n$.

One can now relate $\per (n)$ and $\ord (M,n)$ as follows, which 
provides an efficient way to calculate $\ord (M,n)$.

\begin{prop} \label{prop:ord-versus-per}
  Let\/ $M\in\Mat (2,\ZZ)$ be fixed and let\/ $(u_{m})_{m\ge 0}$ be the
  corresponding recursive sequence from~\eqref{eq:rec-def}. If $n \ge
  2$ is an integer with\/ $\gcd (n,D)=1$, $\ord (M,n)$ divides $\per
  (n)$. Moreover, with\/ $N_{n} := n/\gcd \bigl( n, \mgcd (M) \bigr)$,
  one has
\[
    \ord (M,n) \, = \, \per (N_{n})
\]
  whenever $N_{n} > 1$. In particular, this gives
  $\ord (M,n) = \per (n)$ whenever $n$ and $\mgcd (M)$ are
  coprime.

  In the remaining case, $N_{n}=1$, the matrix satisfies\/ $M \equiv
  \alpha \one\, \bmod{n}$ with $\alpha\in (\ZZ/n\ZZ)^{\times}$, so
  that\/ $\ord (M,n)$ is the order of\/ $\alpha$ modulo $n$.
\end{prop}

\begin{proof}
If $M=\left(\begin{smallmatrix} a & b \\ c & d \end{smallmatrix}
\right)$, the iteration formula \eqref{eq:iter2} implies
that $M^{m} \equiv \one \bmod{n}$ if and only if
\[
    u_{m} a - D u_{m-1} \equiv 1 \, , \quad
    u_{m} b \equiv 0 \, , \quad u_{m} c \equiv 0 \, , 
    \quad \mbox{and} \quad
    u_{m} d - D u_{m-1} \equiv 1 \quad \bmod{n}\ts ,
\]
so that also $u_{m} (a-d) \equiv 0\, \bmod{n}$. Consequently, $n$
divides $u_{m} b$, $u_{m} c$ and $u_{m} (a-d)$. This implies that
$u_{m}$ is divisible by $\frac{n}{\gcd(n,b)}$, $\frac{n}{\gcd(n,c)}$
and $\frac{n}{\gcd(n,a-d)}$, hence also by the least common multiple
of these three numbers, which is the integer
\[
   N_{n} \, = \, \frac{n}{\gcd \bigl(n,\gcd(b,c,a-d)\bigr)}
   \, = \, \frac{n}{\gcd \bigl(n,\mgcd(M)\bigr)}.
\]
Since $N_{n} | n$, we now also have $u_{m} a - D u_{m-1} \equiv 1
\bmod{N_{n}}$. When $u_{m} \equiv 0 \bmod{N_{n}}$, the recursion now gives
$u_{m+1} \equiv T u_{m} - D u_{m-1} \equiv - D u_{m-1} \equiv 1 - u_{n} a
\equiv 1 \bmod{N_{n}}$. Consequently, $M^{m} \equiv\one \bmod{n}$ is
equivalent to $u_{m} \equiv 0$ and $u_{m+1}\equiv 1$ mod $N_{n}$. So, for
$N_{n} > 1$, one has
\[
   \ord (M,n) \, = \, \per (N_{n}) \ts ,
\]
which is the period of the sequence $(u_{m})_{m\ge 0}$ modulo $N_{n}$.
Since $\per (N_{n})$ clearly divides $\per (n)$, one finds
$\ord (M,n)\ts  | \per (n)$.

Finally, when $N_{n} =1$, one has $n \ts | \mgcd (M)$, which implies
$M\equiv \alpha \one \bmod{n}$, where we have $\alpha^{2} \in
(\ZZ/n\ZZ)^{\times}$ due to $\gcd (n,D)=1$. Since this also implies
$\alpha \in (\ZZ/n\ZZ)^{\times}$, the last claim is clear.
\end{proof}

\begin{remark}
  Instead of the characteristic polynomial $P_M$, any other monic
  polynomial that annihilates $M$ can be employed to derive a
  recursive sequence whose period is a multiple of the matrix order
  modulo $n$.  For $n=p$\, a prime, the unique minimal polynomial
  $Q_M$ of $M$ suggests itself to be chosen.  For $d=2$, $Q_M$ has
  smaller degree than $P_M$ precisely when $M = \alpha \one$, whence
  $\mgcd (M) = 0$ and $Q_M (z) = z-\alpha$. Consequently, $\ord (M,
  p)$ is then always equal to the order of $\alpha$ modulo $p$.
  \exend
\end{remark}

\section{Orbit pretail structure of toral endomorphisms}
\label{sec:orbit}

In this section, we look at the action of $M\in \mathrm{Mat}(d,\ZZ)$
on a lattice $L_{n}$, with special emphasis on the structure of general
endomorphisms. When $M$ is not invertible, this manifests itself in
the existence of non-trivial `pretails' to periodic orbits, with rather
characteristic properties. More precisely, given a periodic point $y$
of $M$, a finite set of iterates (or suborbit)
\begin{equation} \label{eq:pretail}
   O \, = \, \{ x, Mx, M^2 x, \ldots , M^n x = y \}
\end{equation} 
is called a \emph{pretail} (of $y$) if $y$ is the only periodic point
of $M$ in $O$.

\subsection{General structure}
Let $M$ and $n$ be fixed, and define $R = \ZZ/n \ZZ$.  Let $\Mper(M)$
denote the set of periodic points on the lattice $\tilde{L}_{n}$,
under the action of $M$ mod $n$. Due to the linear structure of $M$,
$\Mper(M)$ is an $M$-invariant submodule of $\tilde{L}_{n}$. It is the
maximal submodule on which the restriction of $M$ acts as an
invertible map.  The kernel $\ker (M^k)\subset \tilde{L}_{n} $ denotes
the set of points that are mapped to $0$ under $M^k$. One has
$\ker(M^k) \subset \ker (M^{k+1})$ for all $k\ge 0$, and this chain
stabilises, so that $\bigcup_{k\ge 0} \ker(M^{k})$ is another
well-defined and $M$-invariant submodule of $\tilde{L}_{n}$. This is
then the maximal submodule on which the restriction of $M$ acts as a
nilpotent map. Note that $\Mper (M) \cap \ker (M^{k}) = \{ 0 \}$ for
all $k\ge 0$.

Consider an arbitrary $x\in\tilde{L}_{n}$ and its iteration under $M$.
Since $\lvert \tilde{L}_{n}\rvert = n^{d}$ is finite, Dirichlet's
pigeon hole principle implies that this orbit must return to one of
its points. Consequently, every orbit is a cycle or turns into one
after finitely many steps, i.e.\ it is eventually periodic. By
elementary arguments, one then finds the following result.

\begin{fact} \label{f:upperBds}
   There are minimal integers $m\geq 0$ and $k \geq 1$ such that
   $M^{k +m }\equiv M^{m } \bmod n$.  The number $k$ is the
   least common multiple of all cycle lengths on $\tilde{L}_{n}$, 
   while $m$ is the maximum of all pretail lengths.  Clearly,
   $\Mper (M) =\fix(M^{k})$.  \qed
\end{fact}

The lattice $\tilde{L}_{n}=R^d$ is a free $R$-module.  The
modules $\Mper(M)$ as well as $\fix(M^j)$ and $\ker (M^j)$ for
$j\geq 1$ are submodules of it, with $\fix(M^i)\cap\ker(M^j) = \{0\}$
for all $i\ge 1$ and $j\ge 0$. Recalling some results
on modules from \cite[Ch.~III]{Lang} now leads to the following
consequences.
\begin{fact}\label{f:directSum}
   Let $m$ and $k$ be the integers from Fact $\ref{f:upperBds}$.
   If $m\ge 1$, one has 
\[
   \{0\} \subsetneq \ker(M) \subsetneq \ker(M^2) \subsetneq \ldots
   \subsetneq \ker(M^{m}) \subseteq \tilde{L}_{n},
\]
   while $\ker(M^{m+j})= \ker(M^{m})$ for all $j\geq 0$. 
   Moreover, one has
   \[
      \tilde{L}_{n} = \fix(M^{k})\oplus\ker(M^{m}),
   \]
   which is the direct sum of two $M$-invariant submodules.
  Hence, $\Mper (M)$ and $\ker (M^m)$ are finite
  projective $R$-modules.   \qed
\end{fact}

In general, the projective summands need not be free.
As a simple example, let us consider $\tilde{L}^{}_{6}$ with $d=1$ and
$M=2$. Here, $\Mper (M) = \{ 0,2,4 \}$ covers the fixed
point $0$ and a $2$-cycle, while $\ker (M) = \{ 0,3\}$. Both are
modules (and also principal ideals, hence generated by a single
element) over $\ZZ/6\ts \ZZ$, but do not have a basis, hence are not
free. Nevertheless, one has $\ZZ/6\ts \ZZ = \Mper (M) \oplus \ker
(M)$.  We will return to this question below.

\subsection{The pretail tree}
   
Consider the equation $M^\ell x = y$, with $\ell \in \NN$, for some
arbitrary, but fixed $y$ in $\tilde{L}_{n}$.  In general, this
equation need not have any solution $x\in\tilde{L}_{n}$.  On the other
hand, when there is a solution $x\in\tilde{L}_{n}$, the set of all
solutions is precisely $x + \ker(M^\ell)$, which has cardinality
$\lvert \ker(M^\ell) \rvert$.  If $y$ is a periodic point, the first
case can never occur, as there is then at least one predecessor of
$y$. Due to the linearity of $M$, the structure of the set of pretails
of a periodic point $y$ must be the same for all $y \in\Mper(M)$ (note
that there is precisely one predecessor of $y$ in the periodic orbit,
which might be $y$ itself, while all points of the pretail except $y$
are from the complement of the periodic orbit). 

Consequently, we can study the pretail structure for $y=0$. Let us
thus combine all pretails of the fixed point $0$ into a (directed)
graph, called the \emph{pretail graph} from now on; see \cite{WBook}
for general background on graph theory. A single pretail is called
\emph{maximal} when it is not contained in any longer one.  By
construction, there can be no cycle in the pretail graph, while $y=0$
plays a special role.  Viewing each maximal pretail of $0$ as an
`ancestral line', we see that this approach defines a rooted tree with
root $0$. Note that an isomorphic tree also `sits' at every periodic
point $y$.

\begin{figure}
\centerline{\includegraphics[width=0.7\textwidth]{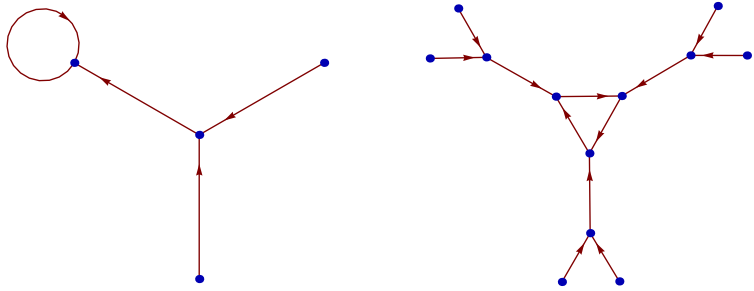}}
\caption[]{The directed graph for the action of
  $M=\left(\begin{smallmatrix}4&0 \\ 1&4\end{smallmatrix}\right)$ on
    the lattice $\tilde{L}^{}_{6}$. The matrix has three fixed points
    and two $3$-cycles (shown once each), while each periodic point
    has the same binary tree of height $2$ as pretail tree.}
  \label{fig:mixed}
\end{figure}

\begin{coro} \label{coro:p-tree}
   Every periodic point of $M$ on $\tilde{L}_{n}$ has a directed
   pretail graph that is isomorphic to that of the fixed point\/ $0$.
   Up to graph isomorphism, it thus suffices to analyse the latter.
   By reversing the direction, it is a rooted tree with root\/ $0$.
   This tree is trivial if and only if $M$ is invertible on
   $\tilde{L}_{n}$.  \qed
\end{coro}

Two illustrative examples are shown in Figures~\ref{fig:mixed} and
\ref{fig:triple}.  Each $M$ defines a unique (rooted) pretail tree on
a given lattice.  If $v_i$ denotes the number of nodes (or vertices)
of this tree with graph distance $i$ from the root, we have
$v^{}_{0}=1$ and
\begin{equation} \label{eq:tree1}
  \big\lvert \ker (M^j) \big\rvert
  \, = \, v^{}_{0} + v^{}_{1} + \cdots + v^{}_{j}
\end{equation}
for all $j\ge 0$, where $v_i = 0$ for all $i$ larger than the
maximal pretail length. Also, one has
\begin{equation} \label{eq:tree2}
   v_{j} \, = \,
   \big\lvert \ker(M^{j}) \setminus
   \ker(M^{j-1}) \big\rvert \, = \,
   \big\lvert \ker(M^{j}) \big\rvert
   - \big\lvert \ker(M^{j-1}) \big\rvert
\end{equation}
for $j\ge 1$, where the second equality follows from the submodule
property. Recall that terminal nodes of a rooted tree (excluding the
root in the trivial tree) are called \emph{leaves}. With this
definition, the total number of leaves on $\tilde{L}_n$ is $\lvert
\tilde{L}_n \setminus M \tilde{L}_n \rvert$. For $i\in \NN$, define
$w_{i}$ to be the number of nodes with graph distance $i$ from the
root that fail to be leaves, and complete this with $w^{}_{0} = 0$ for
the trivial tree and $w^{}_{0}=1$ otherwise. It is clear that this
leads to
\begin{equation} \label{eq:tree3}
   v^{}_{i+1} \, = \, w^{}_{0} 
     \bigl(w^{}_{i} \, \big\lvert
   \ker (M) \big\rvert - \delta^{}_{i,0}\bigr) ,
\end{equation}
via the number of solutions to $Mx=0$ and the special role of the
root, and inductively to
\begin{equation} \label{eq:tree4}
   \big\lvert \ker(M^{i+1}) \big\rvert \, = \,
   (w^{}_{0} + w^{}_{1} + \cdots + w^{}_{i} )
   \big\lvert \ker (M) \big\rvert  
\end{equation}
whenever $w^{}_{0} = 1$, with both relations being valid for all 
$i\ge 0$.

\begin{figure}
\centerline{\includegraphics[width=0.9\textwidth]{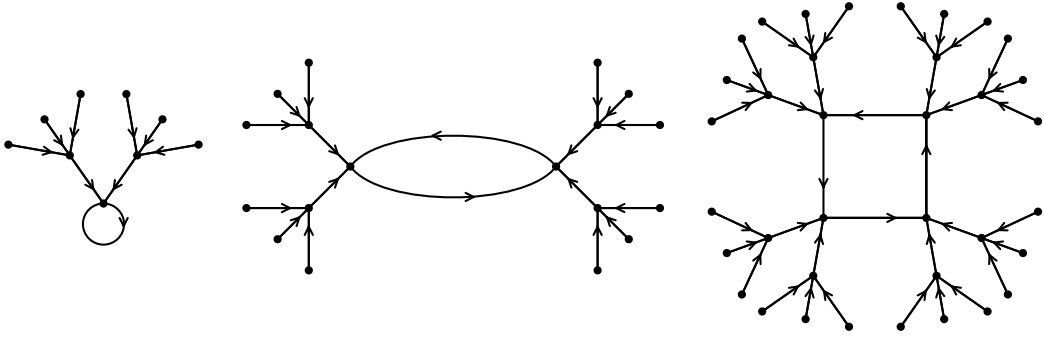}}
\caption[]{The directed graph for the action of
  $M=\left(\begin{smallmatrix}0&12 \\ 1&6\end{smallmatrix}\right)$ on
  the lattice $\tilde{L}^{}_{15}$. The only fixed point of $M$ is $0$,
  while it has two $2$-cycles and five $4$-cycles (shown once each
  only). All pretail trees have the same height.}
  \label{fig:triple}
\end{figure}

\begin{lemma} \label{lem:samelength}
   If\/ $M$ acts on $\tilde{L}_{n}$, its uniquely defined pretail
   tree of the fixed point\/ $0$ has height $m\ge 0$, and the
   following properties are equivalent.
\begin{itemize}
\item[\textrm{(i)}] All maximal pretails have the same length $m$;
\item[\textrm{(ii)}] One has $v_i = w_i\ne 0$ for all\/ $0 \le i < m$ 
     and $w_{i}=0$ for $i\ge m$;
\item[\textrm{(iii)}] One has $\lvert \ker(M^{i+1})\rvert =
     \lvert \ker (M) \rvert \, \lvert \ker (M^{i})
     \rvert = \lvert \ker (M) \rvert^{i+1}$ for all\/
     $0\le i < m$, together with $\lvert \ker (M^{m+j})\rvert
     = \lvert \ker(M^{m}) \rvert$ for all $j\ge 0$.
\end{itemize}
In particular, $m$ is the integer from Fact~$\ref{f:upperBds}$.
\end{lemma}
\begin{proof}
By Corollary~\ref{coro:p-tree}, the pretail tree is trivial (hence
$w^{}_{0}=m=0$) if and only if $M$ is invertible on
$\tilde{L}_{n}$. Since all claims are clear for this case, let us now
assume that $M$ is not invertible on $\tilde{L}_{n}$.

All maximal pretails have the same length if and only if all leaves of
the pretail tree of $0$ have the same graph distance from the root
$0$. Clearly, the latter must be the height $m$ of the tree. When $M$
is not invertible on $\tilde{L}_{n}$, the tree is not the trivial one,
so $m\ge 1$.  The equivalence of (i) and (ii) is then clear, since
both conditions characterise the fact that all leaves have distance
$m$ from the root.

The implication $\mathrm{(ii)}\Rightarrow\mathrm{(iii)}$ can be seen
as follows. The first claim is trivial for $i=0$, as $\ker (M^0) =
\ker (\one) = \{ 0 \}$.  Assuming (ii), Eqns.~\eqref{eq:tree2} --
\eqref{eq:tree4} yield
\[
   \big\lvert \ker (M^{i+1}) \big\rvert \, = \,
   \big\lvert \ker (M^{i}) \big\rvert + 
   w_{i} \big\lvert \ker (M) \big\rvert \, = \,
   \big\lvert\ker (M^i )\big\rvert + \bigl( \lvert \ker{M^i} \rvert - 
   \lvert \ker{M^{i-1}}\rvert \bigr) \big\lvert\ker{M}\big\rvert 
\]
for $1 \le i < m$, which (inductively) reduces to the first condition
of (iii), while the second is clear from the meaning of $m$.

Conversely, the second condition of (iii) means $w_{m+j}=0$
for all $j\ge 0$, while the first condition, together with
Eqns.~\eqref{eq:tree3} and \eqref{eq:tree4}, successively gives
$v_{i} = w_{i}$ for all $0\le i < m$.
\end{proof}

On the lattice $\tilde{L}_{p^r}$, when $\lvert \ker (M)\rvert = p$,
one can say more.

\begin{prop}\label{prop:pre-length}
  Consider the action of $M$ on the lattice $\tilde{L}_{p^r}$. When\/
  $\lvert \ker (M)\rvert = p$, one has $\lvert \ker (M^{i})\rvert =
  p^{\min(i,m)}$ for all\/ $i\ge 0$, where $m$ is the integer from
  Fact~$\ref{f:upperBds}$ for\/ $n=p^r$. This means $v_{i} = p^{i-1}
  (p-1)$ for $1\le i \le m$, and all maximal pretails share the same
  length $m$.
\end{prop}

\begin{proof}
By assumption, $M$ is not invertible, and the last claim is obvious
from Lemma~\ref{lem:samelength} in conjunction with
Eq.~\eqref{eq:tree1}. We thus need to prove the formula for the
cardinality of $\ker (M^{i})$ for arbitrary $i\ge 0$.

Since $0$ is a fixed point of $M$, we clearly have $v^{}_{0} = 1$ and
$v^{}_{1}=p-1$, together with the inequality $0 \le w^{}_{1} \le
p-1$. If $w^{}_{1}=0$, we have $m=1$ and we are done. Otherwise,
$\ker (M) \subsetneq \ker (M^{2})$, hence $\lvert \ker (M^2) \rvert =
p^{j}$ for some $j\ge 2$, as the kernel is a subgroup of our lattice
(which has cardinality $p^{rd}$). This forces $w^{}_{1} = p-1$ and
$j=2$.  More generally, when $\lvert \ker (M^{i}) \rvert = p^{i}$ for
some $1\le i < m$, one cannot have $w_{i}=0$, so that $\lvert \ker
(M^{i+1})\rvert = p^{i+j}$ for some $j\ge 1$.  Since now $0\le w_i
\le p^{i-1} (p-1)$, the only possibility is $j=1$ together with
$w_i = p^{i-1} (p-1)$. This argument can be repeated inductively
until $i=m$ is reached, with $\lvert \ker (M^{m + j})\rvert =
\lvert \ker (M^{m})\rvert$ for all $j\ge 0$.
\end{proof}

\begin{figure}
  \centerline{\includegraphics[width=\textwidth]{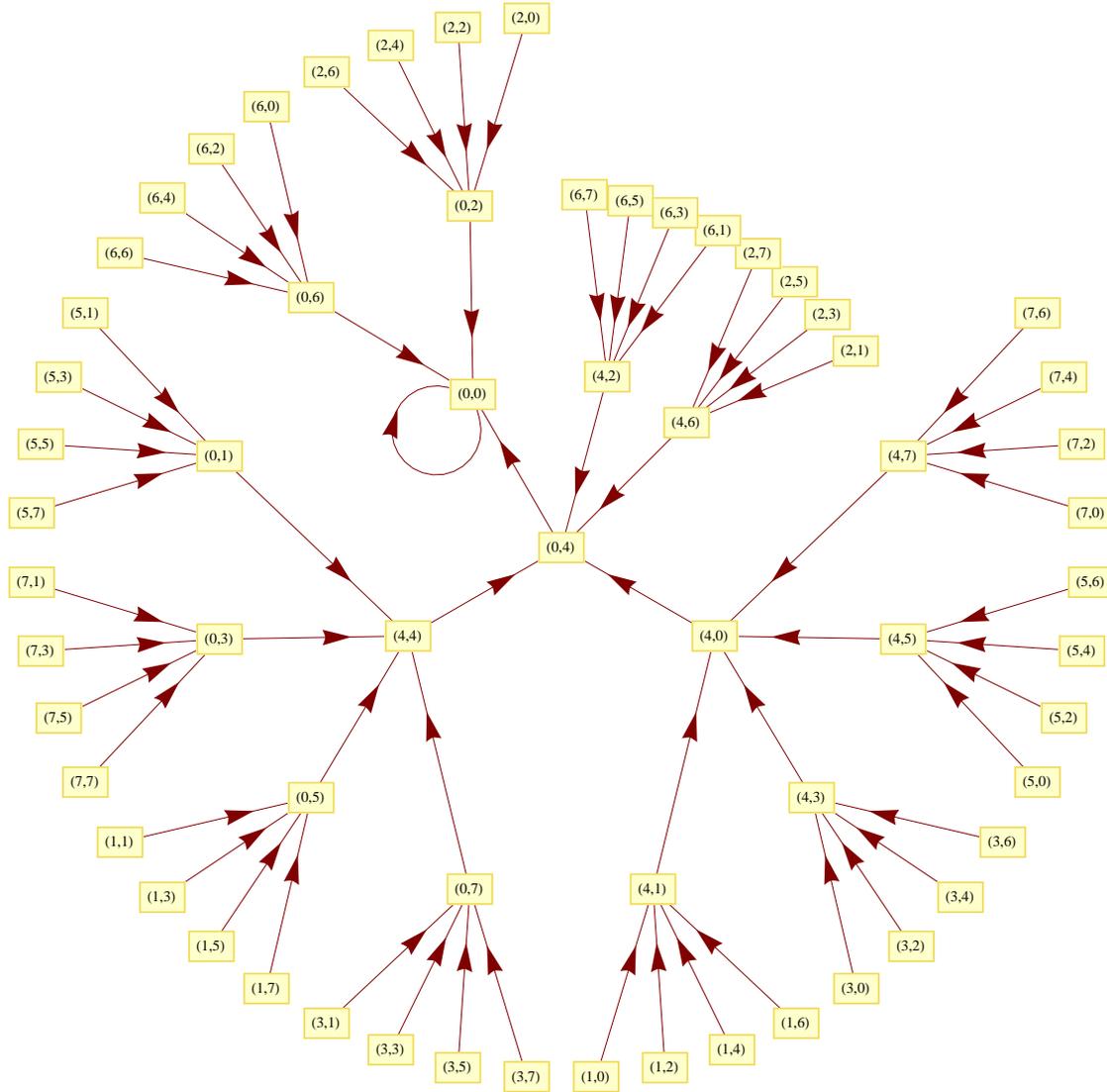}}
  \caption{The pretail graph for Example~\ref{ex:big-tree}, with
  coordinates for the action of the matrix $M$ on $\tilde{L}^{}_{8}$,
  where it is nilpotent with nil-degree $4$.}
\label{fig:big-tree}
\end{figure}

In general, the maximal pretails need not share the same length,
which means that we still have to extend our point of view.

\begin{example} \label{ex:big-tree}
Consider the matrix $M=\left(\begin{smallmatrix} 4 & 4 \\ 1 & 4
\end{smallmatrix}\right)$ on $\tilde{L}^{}_{8}$, where it is nilpotent
(mod $8$) with nil-degree $4$. Since $\card (\ker (M))=4$,
Proposition~\ref{prop:pre-length} does not apply.  The (directed)
pretail graph spans the entire lattice and is shown in
Figure~\ref{fig:big-tree}, together with the loop at $0$ that marks
this point as the root of the tree (which emerges from the figure by
removing this loop and reversing all arrows). \exend
\end{example}

So far, we have looked at a single lattice $\tilde{L}_{n}$. However,
any given matrix $M$ immediately defines a \emph{sequence} of trees
via $\tilde{L}_{n}$ with $n\in\NN$. When $d=2$, the result of
\cite[Thm.~2]{BRW} implies the following result.

\begin{coro}\label{coro:BRW}
  Let $M,M^{\ts\prime}\in\mathrm{Mat} (2,\ZZ)$ be two matrices
  with the same trace, determinant and $\mgcd$. Then, they have
  the same sequence of pretail trees on the lattices $\tilde{L}_{n}$.
  \qed
\end{coro}

\subsection{Decomposition on $\tilde{L}_{p^r}$}

When the integers $u,v$ are coprime, one has $L_{uv} \simeq L_u
\oplus L_v$, wherefore the action on $L_n$ with $n\in\NN$ is
completely determined by that on $L_{p^r}$, for all $p^r || n$. In
particular, the pretail orbit structure on an arbitrary $L_n$ can be
derived from that on the sublattices associated with the factors in
the prime factorisation of $n$. 

Define $R_r = \ZZ/p^r \ZZ$, which is a \emph{local} ring, with unique
maximal ideal $(p) = p \ts R_r$.  The latter contains all zero
divisors.  By \cite[Thm.~X.4.4]{Lang}, we then know that the two
projective modules $\Mper(M)$ and $\ker (M^{k})$ of
Fact~\ref{f:directSum} are \emph{free}, so each has a basis.
Consequently, one knows that the linear map on $\tilde{L}_{p^r}$
defined by $M$ induces unique linear maps on $\fix(M^{k(r)})$ and
$\ker(M^{m(r)})$, and $M$ is conjugate to the direct sum of these
maps, compare \cite[Prop.~4.3.28]{AW}. Each of the latter, in turn,
admits a matrix representation with respect to any chosen basis of the
corresponding submodule. Different choices of bases lead to conjugate
matrices, by an application of \cite[Prop.~4.3.23]{AW}.

\begin{coro}\label{c:blockdiagonal}
   On $\tilde{L}_{p^r}$, $M$ is similar to a block diagonal matrix $\left(
   \begin{smallmatrix} A & 0\\ 0 & B \end{smallmatrix}\right)$ over
   $R_r$, where $A$ is invertible and $B$ is nilpotent, the latter of
   nil-degree $n(B)$ say.  The block matrices $A$ and $B$ are unique up to
   similarity.  The direct sum from Fact~$\ref{f:directSum}$ now reads
\[
      \tilde{L}_{p^r} = \fix(M^{\ord(A,p^r)})\oplus\ker (M^{n(B)}),
\]
   where the concrete form of the exponents $k$ and $m$ of
   Fact~$\ref{f:directSum}$ follows from the block diagonal structure
   of\/ $M$ chosen.  Here, $\fix(M^{\ord(A,p^r)})\simeq
   R_r^{d^{\prime}}$ and $\ker (M^{n(B)})\simeq
   R_r^{d-d^{\ts\prime}}$, where one has $d^{\ts\prime} = \mathrm{rank}
   \left(\Mper (M)\right) \leq d$.

   Furthermore, $d^{\ts\prime}$ is independent of $r$. When comparing the
   above objects as modules over the ring $R_s$ for different $s$, one
   has
\[
   \begin{split}
      \mathrm{rank}^{}_{1}( \Mper^{}_{1}(M)) & = 
      \mathrm{rank}^{}_{r}(\Mper^{}_{r}(M)) 
      = d^{\ts\prime} \quad \text{and} \\
      \mathrm{rank}^{}_{1}( \ker^{}_{1}(M^{m(1)})) & =
      \mathrm{rank}^{}_{r}(\ker^{}_{r} (M^{m(r)}))
      = d-d^{\ts\prime} ,
   \end{split}
\]
where an index $s$ at\/ $\Mper$, $\ker$ or $\mathrm{rank}$ refers to
$R_{s}$ as the underlying ring.
\end{coro}
\begin{proof}
The diagonal block-matrix structure is clear from
\cite[Props.~4.3.28 and 4.3.23]{AW}, while the isomorphism
claim follows from \cite[Cor.~III.4.3]{Lang}. 

For the last claim, observe that $A$ and $B$ can be viewed as integer
matrices acting on $R_{r}^{d^{\ts\prime}}$ and
$R_{r}^{d-d^{\ts\prime}}$, respectively. Here, $B^s = 0 \bmod p^r$ for
some $s\in\NN$ and $\gcd(\det(A),p)=1$, because $A$ is invertible mod
$p^r$ and $\det(A)$ must be a unit in $R_r$. But this means that the
reduction of $A \bmod p$ is also invertible over $R^{}_{1}=\ZZ/p\ZZ$,
while the reduction of $B \bmod p$ is still nilpotent. Consequently,
these reductions provide the blocks for the direct sum over
$\tilde{L}_{p}$, and the claim is obvious.
\end{proof}

Since two free modules of the same rank are isomorphic
\cite[Cor.~III.4.3]{Lang}, we also have the following consequence.
\begin{coro}
   One has the following isomorphisms of $R_1$-modules $($as $\FF_p$-vector
   spaces\/$)$,
   \[
    \Mper_{r}(M)/ p\, \Mper_{r} (M) \simeq \Mper_{1}(M)\  \ \text{and}\ \
    \ker_{r}(M^{m(r)})/ p \ker_{r} (M^{m(r)}) \simeq
    \ker^{(d)}_{1}(M^{m(1)}).
   \]
    This implies
   \[
   \big\lvert\Mper_{r}(M)\big\rvert = p^{r d^{\prime}} 
     = \big\lvert\Mper_{1}(M)\big\rvert^r
   \quad \text{and}\quad 
   \big\lvert\ker_{r}(M^{m(r)})\big\rvert = p^{r (d-d^{\prime})} 
       = \big\lvert\ker_{1}(M^{m(1)})\big\rvert^r  
   \] 
for the cardinalities of the finite modules.\qed
\end{coro}

At this point, it is reasonable to link the properties of $M$
on $\tilde{L}_{p^r}$ to its minimal polynomial over $\FF_{p}$.

\begin{lemma}\label{l:mipoDecomp}
   If $M$ is similar mod $p$ to the block diagonal matrix of
   Corollary~$\ref{c:blockdiagonal}$, its minimal polynomial 
   over $\FF_p$ is $\mu^{}_M (x) = x^s f(x)$, where $f$ is a monic
   polynomial of order $k$ over $\FF_p$ with $f(0)\ne 0$. When
   $M$ is invertible, one has $s=0$ and $k=\gcd \{ \ell\in\NN \mid
   M^{\ell}\equiv\one \bmod p\}$. When $M$ is nilpotent, $f=1$ and
   $s=\gcd\{ t\in\NN \mid M^t \equiv 0 \bmod p\}$. In all remaining
   cases, $s$ and\/ $k$ are the smallest positive integers such that
   $B^s \equiv 0$ and $A^k \equiv \one \bmod p$.
\end{lemma}
\begin{proof}
   Recall from \cite[Def.~3.3.2]{LN} that the order of a polynomial
   $f\in\FF_{p} [x]$ with $f(0)\ne 0$, denoted by $\ord(f,p)$, is the
   smallest positive integer $\ell$ such that $f(x) | (x^{\ell}\! -\!
   1)$. When $M$ is invertible and $k$ as claimed, the polynomial $x^k
   \! -\! 1$ annihilates $M$. Since $\mu^{}_{M} (0) \ne 0$ in our
   case, we have $\mu^{}_{M} = f$ with $f(x) | (x^k \!- \! 1)$, so
   that $\ord(f,p) | k$ by \cite[Lemma~3.3.6]{LN}. By construction,
   $k$ is also the minimal positive integer such that $x^k \! -\! 1$
   annihilates $M$, hence $k=\ord(f,p)$.

   When $M$ is nilpotent, the claim is obvious, because $0$ is then
   the only possible root of the minimal polynomial over $\FF_{p}$, as
   all other elements of the splitting field of $f$ are units.

   In all remaining cases, $M$ is similar to $A\oplus B$ with $A$
   invertible and $B$ nilpotent, by Corollary~\ref{c:blockdiagonal}.
   We thus know that $\mu^{}_{M} (x) | x^s (x^k \! - \! 1)$ with $s$
   and $k$ as claimed, since the latter annihilates both $A$ and $B$.
   Observe that $B^s (B^k \! - \! \one) \equiv 0 \bmod p$ means
   $B^{k+s}\equiv B^s \bmod p$. Since $B$ is nilpotent, its powers
   cannot return to a non-zero matrix, hence $B^s \equiv 0\mod p$.
   Similarly, $A^{s+k}\equiv A^s \bmod p$ is equivalent with $A^k
   \equiv \one \bmod p$, as $A$ is invertible. This shows that we must
   indeed have $\mu^{}_{M} (x) = x^s f(x)$ with $\ord(f,p)=k$.
\end{proof}

\subsection{Classification on $\tilde{L}_{p}$}

When we consider $n=p$, we can go one step further, because $\FF_{p}$
is a field and one can classify nilpotent matrices via their Jordan
normal form. This follows from the observation that $0$ is the only
possible eigenvalue. Recall that an \emph{elementary shift matrix} is
an upper triangular matrix with entries $1$ on the upper
super-diagonal and $0$ everywhere else (this includes the $0$-matrix
in one dimension). An elementary shift matrix is nilpotent, with
nil-degree equal to its dimension. The following result is now
a standard consequence of the Jordan normal form over fields
\cite{J,Lang}.
\begin{fact}
  The nilpotent matrices in\/ $\mathrm{Mat} (d,\FF_{p})$ are conjugate
  to block-diagonal matrices, where each block is an elementary
  shift matrix.  \qed 
\end{fact}

Some of this structure survives also for general $n$. For instance,
the $0$-matrix in dimension $d\ge 1$ leads to the regular $(n^{d} -
1)$-star as its pretail tree on $\tilde{L}_{n}$.  When $d\ge 2$, the
$d$-dimensional elementary shift matrix, on $\tilde{L}_{n}$, results in a
semi-regular tree with $w^{}_{0} = 1$, $w^{}_{1} = n-1$ and
  $w_{i} = n$ for $2 \le i \le d-1$, while $w_{j}=0$ for all
  $j\ge d$. These trees have the property that all maximal pretails
share the same length, which is the nil-degree of the matrix.
One can now go through all possible block-diagonal combinations of
such elementary shift matrices. This is a combinatorial problem and
gives the possible pretail trees over $\FF_{p}$.

As already suggested by Proposition~\ref{prop:pre-length}, the
structure of $\ker (M)$ plays an important role for the structure of
the pretail tree.  Together with the linearity of $M$, it constrains
the class of trees that are isomorphic to the pretail tree of some
integer matrix.  A more detailed analysis is contained in \cite{Diss}.

\section{Symmetry and reversibility}\label{sec:revsym}

Reversibility is an important concept in dynamics, compare \cite{RQ}
and references therein for background, and \cite{Devog} for an early
study in continuous dynamics. Here, we focus on discrete dynamics,
as induced by toral auto- and endomorphisms.

A matrix $M$ is called \emph{reversible}, within a given or specified
matrix group $\cG$, if it is conjugate to its inverse within
$\cG$. Clearly, this is only of interest when $M^{2} \ne
\one$. To put this into perspective, one usually defines
\[
   \cS (M) = \{ G \in \cG \mid GMG^{-1} = M \}
   \quad \mbox{and} \quad
   \cR (M) = \{ G \in \cG \mid GMG^{-1} = M^{\pm 1} \}
\]
as the \emph{symmetry} and \emph{reversing symmetry} groups of $M$;
see \cite{BRgen} and references therein for background and
\cite{BRcat,BRtorus} for examples in our present context. In
particular, one always has $\cR (M) = \cS (M)$ when $M^2 = \one$
or when $M$ is not reversible, while $\cR (M)$ is an extension
of $\cS (M)$ of index $2$ otherwise.

Note that a nilpotent matrix $M$ (or a matrix with nilpotent summand,
as in Corollary~\ref{c:blockdiagonal}) cannot be reversible in this
sense. However, they can still possess interesting and revealing
symmetry groups, although it is more natural to look at the \emph{ring}
of matrices that commute with $M$ in this case.

\begin{example}\label{ex:tree-sym}
Reconsider the matrix $M=\left( \begin{smallmatrix} 4 & 4 \\ 1 & 4
\end{smallmatrix}\right)$ from Example~\ref{ex:big-tree}, and its action 
on $\tilde{L}^{}_{8}$. Clearly, $M$ commutes with every element of the
ring $\ZZ/8 \ts \ZZ\, [M]$, which contains $64$ elements. This follows
from the existence of a cyclic vector, but can also be checked by a
simple direct calculation. Consequently, the symmetry group (in our
above sense) is the intersection of this ring with $\GL (8, \ZZ/8 \ts
\ZZ)$, which results in
\[
   \cS (M) \, = \, \big\langle \left(
    \begin{smallmatrix} 1 & 4 \\ 1 & 1 \end{smallmatrix} \right) ,
    3 \!\cdot\!\! \mbox{\large $\one$}, 
    5 \! \cdot\!\! \mbox{\large $\one$} \big \rangle
    \, \simeq \, C_{8} \times C_{2} \times C_{2} \ts , 
\]
which is an Abelian group of order $32$.  The matrices in $\cS (M)$
have either determinant $1$ or $5$, with $\{ A \in \cS (M) \mid \det
(A) = 1 \} \simeq C_{4} \times C_{2} \times C_{2}$.

One can now study the action of $\cS (M)$ on the pretail graph of
Figure~\ref{fig:big-tree}, which actually explains all its symmetries.
\exend
\end{example}

In what follows, we derive certain general properties, where we
focus on the reversing symmetry group, with invertible matrices
$M$ in mind.

\subsection{Reversibility of $\SL (2,\ZZ)$-matrices mod $n$}

Recall the matrix mgcd from Eq.~\eqref{eq:mgcd-def}, which is
a conjugation invariant. It can be used to solve the reversibility
at hand as follows.
\begin{theorem} \label{thm:local-rev}
   Let\/ $M\in \SL (2,\ZZ)$ and $n\in\NN$ be arbitrary. Then, the
   reduction of\/ $M$ mod $n$ is conjugate to its inverse within the
   group\/ $\GL (2, \ZZ/n\ZZ)$. The action mod\/ $1$ of any $M\in \SL
   (2,\ZZ)$ on $L_n$ is thus reversible for all $n\in\NN$.

   Moreover, if\/ $M\in \SL (2,\ZZ)$ has\/ $\mgcd(M) = r \ne 0$,
   its reduction mod $n$, for every $n\in\NN$, possesses an involutory 
   reversor.
\end{theorem}
\begin{proof}
  When $M\in\SL (2,\ZZ)$, also its inverse is in $\SL(2,\ZZ)$, and $M$
  and $M^{-1}$ share the same determinant and trace.  Moreover, they
  also have the same $\mgcd$, so that the first claim follows from
  \cite[Thm.~2]{BRW} (or from Lemma~\ref{lem:mgcd}). This immediately
  implies, for all $n\in\NN$, the reversibility of the action mod $n$
  of $M$ on the lattice $\tilde{L}_{n}$, so that the statement on the
  equivalent action of $M$ mod $1$ on $L_n$ is clear.

Now, let $M=\left(\begin{smallmatrix} a & b \\ c & d \end{smallmatrix}
\right)\in \SL (2,\ZZ)$, so that $M^{-1} = \left(\begin{smallmatrix} d
& -b \\ -c & a \end{smallmatrix}\right)$, and $M$ and $M^{-1}$ share
the same determinant ($1$), trace ($a+d$) and mgcd ($r$). Assume $r\ne
0$, let $n\ge 2$ be fixed and consider the matrices mod $n$. Recall
the normal forms
\[
   N (M) = \begin{pmatrix} a & \frac{bc}{r} \\ r & d
       \end{pmatrix}  \quad \text{and} \quad
   N (M^{-1}) = \begin{pmatrix} d & \frac{bc}{r} \\ r & a
       \end{pmatrix} ,
\]
as defined in the proof of \cite[Prop.~6]{BRW}, and note that they are
not inverses of each other. However, by \cite[Prop.~5]{BRW}, there is
some matrix $P_{n} \in \GL (2, \ZZ/n\ZZ)$ with $M = P_{n}^{} N(M)
P_{n}^{-1}$, hence we also have $M^{-1} = P_{n}^{} \bigl(
N(M)\bigr)^{-1} P_{n}^{-1}$. Observe next that
\[
   \bigl( N(M)\bigr)^{-1} = \begin{pmatrix} d & - \ts \frac{bc}{r} 
       \\ - r & a \end{pmatrix}  
   = C \begin{pmatrix} d & \frac{bc}{r} \\ r & a
       \end{pmatrix} C^{-1}
   = C N (M^{-1}) C^{-1} ,
\]
where $C = \left( \begin{smallmatrix} 1 & 0 \\ 0 & -1 \end{smallmatrix}
\right)$ is an involution. On the other hand, $N(M)$ and
$N(M^{-1})$ satisfy the assumptions of \cite[Prop.~6]{BRW}, so that
\[
   N(M^{-1}) = A N(M) A^{-1} \quad \text{with} \quad
   A = \begin{pmatrix} 1 & \frac{d-a}{r} \\ 0 & 1 
   \end{pmatrix} ,
\]
where we globally have $A=\left(\begin{smallmatrix}
1 & 0 \\ 0 & -1 \end{smallmatrix}\right)$ whenever $d=a$
in the original matrix $M$.
Together with the previous observation, this implies
$\bigl( N(M)\bigr)^{-1} = (CA)\ts  N(M) \ts (CA)^{-1}$ where
\[
   CA = \begin{pmatrix} 1 & \frac{d-a}{r} \\ 0 & -1
   \end{pmatrix}
\]
is an involution. Putting everything together, we have
\[
   M^{-1} = \bigl(P_{n}^{} (CA) P_{n}^{-1} \bigr) M
   \bigl(P_{n}^{} (CA) P_{n}^{-1} \bigr)^{-1},
\]
which is the claimed conjugacy by an involution (which
depends on $n$ in general).
\end{proof}

Note that the matrix $M$ in Theorem~\ref{thm:local-rev} need not be
reversible in $\GL(2,\ZZ)$, as the example $M=\left(
\begin{smallmatrix} 4 & 9 \\ 7 & 16 \end{smallmatrix} \right)$ from
\cite[Ex.~2]{BRcat} shows. Nevertheless, for any $M\in \SL (2,\ZZ)$
with $\mgcd(M)\ne 0$ and $n\ge 2$, the (finite) reversing symmetry group
of $M$ within $\GL (2, \ZZ/n\ZZ)$ is always of the form $\cR (M) = \cS(M)
\rtimes C_2$, with $C_2$ being generated by the involutory
reversor. The structure of $\cS(M)$ remains to be determined.

In the formulation of Theorem~\ref{thm:local-rev}, we have focused on
matrices $M\in\SL (2,\ZZ)$ because the condition $\trace(M) =
\trace(M^{-1})$ for a matrix $M$ with $\det(M)=-1$ forces
$\trace(M)=0$, which means that $M$ is itself an involution (and thus
trivially reversible in $\GL(2,\ZZ)$). More interesting (beyond
Theorem~\ref{thm:local-rev}) is the question which matrices $M\in \Mat
(2,\ZZ)$, when considered mod $n$ for some $n\in\NN$, are reversible
in $\GL (2, \ZZ/n\ZZ)$.  Let us begin with $n=p$ being a prime, where
$\ZZ/p\ZZ\simeq \FF_{p}$ is the finite field with $p$ elements.

\subsection{Reversibility in $\GL(2,\FF_{p})$}\label{sec:4.2}
Let us consider the symmetry and reversing symmetry group of an
element of $\GL(2,\FF_{p})$ with $p$ prime, the latter being a group
of order
\[
   \lvert \GL(2,\FF_{p}) \rvert \, = \, 
   (p^{2} - 1)(p^{2} - p) \, = \,
   p \ts (p-1)^{2} (p+1) \ts ,
\]
compare Eq.~\eqref{eq:GL-Fp}. For our further discussion, it is
better to distinguish $p=2$ from the odd primes. For convenience, we
summarise the findings also in Table~\ref{tab:finite-field}.

\begin{example}
For $p=2$, one has $\GL(2,\FF_{2}) = \SL(2,\FF_{2})\simeq D_{3}$, the
latter denoting the dihedral group of order $6$. There are now three
conjugacy classes to consider, which may be represented by the
matrices $\one$, the involution $R= \left( \begin{smallmatrix} 0 & 1
\\ 1 & 0 \end{smallmatrix}\right)$, and the matrix
$M=\left( \begin{smallmatrix} 1 & 1 \\ 1 & 0
\end{smallmatrix}\right)$ of order $3$. The corresponding cycle 
structure on $L_{2}$ is encapsulated in the generating polynomials
$Z_{2} (t)$. They read
\[
   (1-t)^{4} \, , \quad (1-t)^{2} (1-t^2)
   \quad \mbox{and} \quad
   (1-t) (1-t^{3}) \ts ,
\]
respectively, and apply to entire conjugacy classes of
matrices. 

For the (reversing) symmetry groups, one clearly has $\cR (\one) = \cS
(\one) = \GL(2,\FF_{2})$, while $\cR (R) = \cS (R) = \langle R \ts
\rangle \simeq C_{2}$. The only nontrivial reversing symmetry group
occurs in the third case, where $\cS (M) = \langle M \ts \rangle
\simeq C_{3}$. Since $RMR = M^{2} = M^{-1}$, one has $\cR (M) =
\GL(2,\FF_{2}) \simeq C_{3} \rtimes C_{2}$. So, all elements of
$\GL(2,\FF_{2})$ are reversible, though only $M$ and $M^2$ are
nontrivial in this respect.
\exend
\end{example}

\medskip For $p$ an odd prime, one can use the normal forms
for $\GL(2,\FF_{p})$, see \cite[Ch.~XVIII.12]{Lang},
to formulate the results; compare Table~\ref{tab:finite-field}. We
summarise the reversibility and orbit structure here, but omit proofs
whenever they emerge from straight-forward calculations.

\smallskip
I. The first type of conjugacy class is represented by matrices $M= a
\one$ with $a \in \FF_{p}^{\times} \simeq C_{p-1}$. The order of $M$
coincides with the order of $a$ mod $p$, $\ord (a,p)$, which divides
$p-1$.  One clearly has $\cR (M) = \cS (M) = \GL (2,\FF_{p})$ in this
case, either because $a^2=1$ (so that $M=M^{-1}$) or because
$a^2\ne 1$ (so that no reversors are possible).
The corresponding orbit structure on $L_{p}$ comprises one
fixed point ($x=0$) together with $\frac{p^{2} - 1}{\ord (a, p)}$
orbits of length $\ord (a,p)$. The non-trivial orbits starting from
some $x\ne 0$ must all be of this form, as $x$ gets multiplied by $a$
under the action of $M$ and returns to itself precisely when $a^k=1$,
which first happens for $k=\ord(a,p)$.

\smallskip
II. The next type of conjugacy class is represented by matrices
$M=\left(\begin{smallmatrix} a & 1 \\ 0 & a\end{smallmatrix} \right)$
with $a\in\FF_{p}^{\times}$. Its symmetry group is given by
\[
   \cS (M) \, = \, \left\{
   \left(\begin{smallmatrix} \alpha & \beta \\
     0 & \alpha \end{smallmatrix} \right) \big | \,
     \alpha \in \FF_{p}^{\times} , \, \beta \in \FF_{p}^{}
    \right\} \, \simeq \, C_{p} \times C_{p-1} \ts ,
\]
which is Abelian. As generators of the cyclic groups, one can choose
$\left( \begin{smallmatrix} 1 & 1 \\ 0 & 1 \end{smallmatrix} \right)$,
which has order $p$ in $\GL (2,\FF_{p})$, and $\gamma \one$, with
$\gamma$ a generating element of $\FF_{p}^{\times}$.  The reversible
cases are precisely the ones with $a^2 = 1$ in $\FF_{p}$, hence with
$\det (M) = 1$.  Here, $R=\diag (1,-1)$ is a possible choice for the
(involutory) reversor, so that $\cR (M) = \cS (M) \rtimes \langle R
\ts \rangle \simeq (C_{p} \times C_{p-1})\rtimes C_{2}$.

A matrix $M$ of type II (in its normal form as in
Table~\ref{tab:finite-field}) satisfies 
\[
   M^k =\begin{pmatrix}a^k & k a^{k-1}\\ 0 &
   a^k\end{pmatrix}   \qquad \text{for $k\ge 0\,$,}
\]
whence a point $(x,0)$ with $x\ne 0$ is fixed by $M^k$ if and only if
$k=\ord(a,p)$, and a point $(x,y)$ with $x y \ne 0$ if and
only if $p | k$ and $\ord(a,p) | k$.  Since $\ord(a,p)|(p-1)$, one has
$\lcm(p,\ord(a,p))=1$, wherefore this gives $\frac{p-1}{\ord(a,p)}$
orbits of length $p-1$ and $\frac{p\cdot (p-1)}{p\cdot \ord(a,p)} =
\frac{p-1}{\ord(a,p)}$ orbits of length $p \ord(a,p)$ in total.

\smallskip
III. The third type of conjugacy class is represented by
$M = \diag (a,b)$ with $a,b \in \FF_{p}^{\times}$ and $a\ne b$.
This results in $\cS (M) = \{ \diag (\alpha, \beta) \mid
\alpha, \beta \in \FF_{p}^{\times} \} \simeq C_{p-1}^{2}$.
The condition for reversibility leads either to
$a^2 = b^2 = 1$, hence to $b=-a$, or to $ab=1$.
In the former case, $M$ itself is an involution,
so that $\cR (M) = \cS (M)$ is once again the trivial
case, while $\det (M) = ab=1$ leads to genuine reversibility,
with involutory reversor $R=\left(\begin{smallmatrix} 
0 & 1 \\ 1 & 0 \end{smallmatrix} \right)$ and hence 
to $\cR (M) = \cS (M) \rtimes C_{2}$.

{}For a type III matrix, one has $M^k (x,y)^t = (a^k x, b^k y)^t$, so
each of the $p\! - \! 1$ non-zero points $(x,0)^t$ is fixed by
$M^{\ord(a,p)}$; analogously, each of the $p\! - \! 1$ non-zero points
$(0,y)^t$ is fixed by $M^{\ord(b,p)}$.  The remaining points that
are non-zero in both coordinates have period
$\lcm(\ord(a,p),\ord(b,p))$.  In summary, this gives one fixed point,
$\frac{p-1}{\ord(a,p)}$ orbits of length $\ord(a,p)$,
$\frac{p-1}{\ord(b,p)}$ orbits of length $\ord(b,p)$, and
$\frac{(p-1)^2}{\lcm(\ord(a,p),\ord(b,p))}$ orbits of length
$\lcm(\ord(a,p),\ord(b,p))$.

\begin{table}\label{tab:finite-field}
  \caption{Summary of conjugacy structure for $\GL (2,\FF_{p})$ via
    normal forms.  Note that class III is absent for $p=2$. 
    The second possibility for $\mathcal{R} (M)$ always applies
    when $\det (M)=1$. Only non-trivial orbits are counted.}
\renewcommand{\arraystretch}{1.3}
\begin{tabular}{|l|cccc|}
\hline
class & I & II & III & IV \\ \hline 
normal form & $a\ts \one$ 
& $\left(\begin{smallmatrix}a & 1\\
  0 &a \end{smallmatrix}\right)$
& $\left(\begin{smallmatrix}a & 0\\ 0 &b
\end{smallmatrix}\right)$ &
$\left(\begin{smallmatrix}0 & -D\\ 1 &T
\end{smallmatrix}\right)$ \\ 
of matrix class &
$a\in\mathbb{F}_{p}^{\times}$ & 
$a\in\mathbb{F}_{p}^{\times}$ &
$a \ne b\in \mathbb{F}_{p}^{\times}$ & $z^2 - Tz + D$
irred. \\ 
\hline
min.\ polynomial & $(z-a)$ & $(z-a)^2$ & $(z-a)(z-b)$ & $z^2 - Tz +
D$ \\ \hline size of class & $1$ & $p^2 -1$ & $p^2 + p$ & $p^2 - p$ \\ 
\hline 
no.\ of classes & $p-1$ & $p-1$ & $\frac{1}{2}(p-1)(p-2)$ &
$\frac{1}{2}\ts  p \ts (p-1)$\\ 
\hline 
$\mathcal{S}(M)$ & $\GL(2,\mathbb{F}_p)$
& $C_p\times C_{p-1}$ & $C_{p-1}\times C_{p-1}$ & $C_{p^2-1}$\\ 
\hline
$\mathcal{R}(M)$ & $\mathcal{S}(M)$ & $\mathcal{S}(M)$ or &
$\mathcal{S}(M)$ or & $\mathcal{S}(M)$ or \\ & & 
$\mathcal{S}(M)\rtimes C_2$ &
$\mathcal{S}(M)\rtimes C_2$ & $\mathcal{S}(M)\rtimes C_2$ \\ 
\hline
orbit length & $\ord(a,p)$ & see text & see text &
$\ord(\chi^{}_M,p)$\\ 
\hline 
orbit count\raisebox{0ex}[3ex][1.8ex]{} &
$\frac{p^2-1}{\ord(a,p)}$& see text & see text &
$\frac{p^2-1}{\ord(\chi^{}_M,p)}$ \\ 
\hline
\end{tabular}
\end{table}

\smallskip
IV. Finally, the last type of conjugacy class can be represented by
companion matrices of the form $\left(\begin{smallmatrix} 0 & -D \\ 1
& T \end{smallmatrix} \right)$ with the condition that the
characteristic polynomial $z^2 - T z + D$ is irreducible over
$\FF_{p}$.  The determinant and the trace satisfy $D = \eta \eta^{\ts
\prime}$ and $T=\eta + \eta^{\ts \prime}$, where $\eta$ and $\eta^{\ts
\prime}$ are not in $\FF_{p}$, but distinct elements of the splitting
field of the polynomial, which can be identified with $\FF_{p^2}$. One
consequence is that $1+D \pm T = (1\pm\eta)(1\pm\eta^{\ts
\prime}) \ne 0$.

The symmetry group is $\cS (M) = \{ \alpha \one + \gamma M \mid
\alpha, \gamma \in \FF_{p}, \mbox{ not both $0$} \}$, which is an
Abelian group with $p^2 - 1$ elements. The order follows from the
observation that $\det (\alpha \one + \gamma M ) = (\alpha + \gamma
\eta) (\alpha + \gamma \eta^{\ts\prime})$ vanishes only for $\alpha =
\gamma = 0$ in this case. In fact, one has $\cS (M) \simeq C_{p^2 -
1}$, as any matrix $\left(\begin{smallmatrix} 0 & -\eta
\eta^{\ts\prime} \\ 1 & \eta + \eta^{\ts\prime}
\end{smallmatrix}\right) \in \GL(2,\FF_{p})$ with $\eta \in \FF_{p^2}
\setminus \FF_{p}$ has order $p^2-1$ or possesses a root in
$\GL(2,\FF_{p})$ of that order. This relies on the facts that we can
always write $\eta = \lambda^m$, where $\lambda$ is a generating
element of $\FF_{\! p^2}^{\times}\simeq C_{p^2-1}$, and that $\lambda
\lambda^{\prime}$ and $\lambda + \lambda^{\prime}$ are in $\FF_{p}$.
This is a special case of Fact~\ref{f:esp} below and of a statement on
the existence of roots in $\GL(d,\ZZ)$; see Lemma~\ref{lem:root} below.

The condition for reversibility, in view of the above restriction
on $D$ and $T$, can only be satisfied when $D=1$, in which case
$R=\left(\begin{smallmatrix} 0 & 1 \\ 1
& 0 \end{smallmatrix} \right)$ turns out to be an involutory
reversor, so that again $\cR (M) = \cS (M) \rtimes C_{2}$ 
in this case.

Matrices with irreducible characteristic polynomial $\chi^{}_M$
produce orbits of one length $r$ only, where $r$ is the smallest
integer such that $\chi^{}_M (z) | (z^r -1)$, or, equivalently, the
order of its roots in the extension field $\mathbb{F}^{}_{p^2}$.

\smallskip

Putting these little exercises together gives the following result.
\begin{theorem} \label{thm:mod-rev} 
  A matrix\/ $M\in\GL (2,\FF_{p})$ is reversible within this group if
  and only if\/ $M^2 = \one$ or\/ $\det (M) = 1$. Whenever\/ $M^2 =
  \one$, one has\/ $\cR (M) = \cS (M)$.  If\/ $\det (M)=1$ with\/
  $M^2\ne \one$, there exists an involutory reversor, and one has\/
  $\cR (M) = \cS (M) \rtimes C_{2}$.  \qed
\end{theorem}

\begin{remark}
Since $\FF_{p}$ is a field, we can use the following dichotomy to
understand the structure of $\cS(M)$, independently of the chosen
normal forms. A matrix $M\in \GL (2,\FF_{p})$ is either a multiple of
the identity (which then commutes with every element of $\Mat
(2,\FF_{p})$) or it possesses a cyclic vector (meaning an element $v
\in \FF_{p}^{2}$ such that $v$ and $Mv$ form a basis of
$\FF_{p}^{2}$). In the latter case, $M$ commutes precisely with the
matrices of the ring $\FF_{p} [M]$, and we have $\cS(M) = \FF_{p}
[M]^{\times} \! = \FF_{p} [M] \cap \GL (2,\FF_{p})$. This systematic
approach provides an alternative (but equivalent)
parametrisation of the above results for the normal forms.  \exend
\end{remark}

The question for reversibility in $\GL (2,\ZZ/n\ZZ)$ with general $n$
is more complicated. The matrix $M=\left(\begin{smallmatrix} 0 & -4 \\
    1 & 0 \end{smallmatrix}\right)$ is reversible over $\ZZ/3\ZZ$
(where it is an example of type IV), but fails to be reversible over
$\ZZ/9\ZZ$, as one can check by a direct computation. Here, zero
divisors show up via non-zero matrices $A$ with $AM = M^{-1}A$, but
all of them satisfy $\det (A) \equiv 0$ mod $9$. In fact, one always
has $A(L_{9}) \subset L_{3}$ here.

In general, the relation $A\ts MA^{-1} = M^{-1}$ with
$A,M\in \GL (2,\ZZ/n\ZZ)$ implies $MA\ts M=A$ and hence $\det (M)^2 = 1$,
because $\det (A) \in (\ZZ/n\ZZ)^{\times}$.  Over $\FF_{p}$, this gives
$\det (M) = \pm 1$, with reversibility precisely for $\det (M) = 1$
according to Theorem~\ref{thm:mod-rev}.  In general, one has further
solutions of the congruence $m^2 \equiv 1 \bmod{n}$, such as $m=3$ for
$n=8$ or $m=4$ for $n=15$.

In any such case, $M=\left(\begin{smallmatrix} 0 & -m \\
    1 & 0 \end{smallmatrix}\right)$ is a matrix with $M^{2} = - m \ts
\one $. When $m\not\equiv -1 \bmod{n}$, $M$ is of order $4$ in $\GL
(2,\ZZ/n\ZZ)$. It is easy to check that $RMR = M^{-1}=
\left(\begin{smallmatrix} 0 & 1 \\ -m & 0 \end{smallmatrix}\right)$ in
$\GL (2,\ZZ/n\ZZ)$, with
the involution $R=\left(\begin{smallmatrix} 0 & 1 \\
    1 & 0 \end{smallmatrix}\right)$. This establishes reversibility
with $\cR (M) = \cS (M) \rtimes C_{2}$.

\subsection{Some extensions to higher dimensions}\label{sec:4.3}
In principle, a similar reasoning, based on a normal form approach,
can be applied to arbitrary dimensions.  Over the finite field
$\mathbb{F}_p$, normal forms are given by the rational canonical form
and the elementary divisor normal form (`First' and `Second natural
normal form' in the terminology of \cite[\text{\S}6]{Gant}), which are
block diagonal matrices with companion matrices on the diagonal.

The advantage of dealing with companion matrices is that one can
employ the theory of linear recursions: there is a one-to-one
correspondence between the cycle lengths modulo $n\in\NN$ of a certain
initial condition $\boldsymbol{u}=(u_0,\ldots,u_{d-1})$ under the
recursion induced by the polynomial $f$, and the period of the
corresponding point $\boldsymbol{u}^t$ under the matrix iteration of
$C_f$; compare the final remark in \cite{Z}, and Section~\ref{sec:pim}.

Working with a block diagonal matrix of this shape, the analysis can
be done block-wise; in particular, the symmetry groups are the direct
product of the symmetry groups of the component matrices on the
diagonal, augmented by all additional symmetries that emerge from
equal blocks, which can be permuted.

Determining the period lengths associated with irreducible polynomials
amounts to finding their orders in the sense of \cite[Def.~3.3.2]{LN}.
The periods and their multiplicities arising from the powers of
irreducible polynomials that show up in the factorisation of the
invariant factors (the elementary divisors) are then given by
\cite[Thm.~4]{Z}.
    
Extending the analysis to matrices over the local rings $\ZZ/p^r \ZZ$
is more difficult.  In general, it seems hard to write down an
exhaustive system of normal forms for the similarity classes, and to
decide whether given matrices are similar.  However, a solution for a
large subclass of square matrices over the $p$-adic integers $\ZZ_p$
and the residue class rings $\ZZ/ p^r \ZZ$ is presented in \cite{D}.
For a polynomial $f\in\ZZ_p[x]$ whose reduction modulo $p$ has no
multiple factors, a complete system of $d\!\times\! d$ matrix
representatives $X$ with respect to similarity that satisfy
$f(X)\equiv 0 \mod p^r$ is given by all direct sums of companion
matrices which are in agreement with the factorisation of $f$ mod $p$.
For instance, if the reduction of the common characteristic polynomial
modulo $p$ of two matrices does not have any quadratic factors, the
matrices are conjugate mod $p^r$ if and only if they are conjugate mod
$p$ \cite[Thm.~3 and Corollary]{D}.

An exhaustive treatment of conjugacy classes of $3\!\times\! 3$ matrices
over an arbitrary local principal ideal ring can be found in
\cite{AOPV}.

\begin{remark} 
  In \cite{AOPV}, it is pointed out that $2\times 2$ matrices over a
  local ring can be decomposed into a scalar and a cyclic part.  Over
  $\ZZ/ p^r \ZZ$, this decomposition reads
    \[
	M = d\one + p^{\ell} C,
    \]
    where $p^{\ell} = \gcd(\mgcd(M),p^r) = p^{v_p(\mgcd(M))}$ with
    $v_{p}$ denoting the standard $p$-adic valuation, unique
    $d\in \{\sum_{j=0}^{\ell-1} a_j p^j\mid p\nmid a_j\}$ and cyclic
    $C\in\mathrm{Mat}(2,\ZZ/p^{r-\ell}\ZZ)$, which is unique up to
    similarity.  Moreover, $C$ can be chosen as a companion matrix
    with the appropriate trace and determinant.

    Since $d\one$ and $C$ commute, powers of $M$ can be expanded
    via the binomial theorem. Using that the binomials satisfy
    $\frac{n}{\gcd(n,k)}|\binom{n}{k}$, the period
    $\mathrm{per}(x,p^r)$ of all $x\in\mathrm{L}_{p^r}$ is bounded by
    \[
	\mathrm{per}(x,p^r) \leq \ord(d,p^r) \cdot p^{r-\ell},
    \]
    provided that $1\le \ell \le r$.  Let $\Pi_j : \ZZ/ p^r \ZZ
    \rightarrow \ZZ/ p^j \ZZ$ denote the canonical projection, and let
    $\mathcal{S}_j(A)$ be the 
    symmetry group of an integer matrix $A$, viewed as a matrix over
    $\ZZ/ p^j \ZZ$.  Then, for $p\not= 2$ and $\ell\geq 1$, one
    obtains $\mathcal{S}_r (M) = \Pi_{\ell} ^{-1} (\mathcal{S}_{\ell}(C))$
    from the symmetry equations. \exend
\end{remark}

\subsection{Reversibility mod $n$}\label{sec:4.4}

Let $M$ be a general integer matrix, with determinant $D$.  

\begin{fact}\label{Dsquare}
    If\/ $M \in \Mat (d,\ZZ)$ is reversible mod\/ $n$,
    one has $D^2\equiv 1$ mod\/ $n$. Moreover, reversibility
    for infinitely many $n$ implies\/ $D = 1$ or\/ $D = -1$.
\end{fact}
\begin{proof}
The reversibility equation yields $\det M \equiv \det M^{-1}$, hence
$D^2 \equiv 1\mod n$. If $D^2 -1$ has infinitely many divisors, one
has $D^2 = 1$, hence $D = 1$ or $D = -1$.
\end{proof}

Before we continue with some general result, let us pause
to see what Fact~\ref{Dsquare} specifically implies for $d=2$.
\begin{fact}\label{tr0}
   If\/ $M \in \Mat (2,\ZZ)$ with $D\equiv -1$ mod\/ $n$ is
   reversible mod\/ $n$, one has\/ $2\ts \trace (M) \equiv 0$ 
   mod\/ $n$.  In particular, $\trace (M) \equiv 0$ mod\/ $n$ 
   holds whenever\/ $n$ is odd.
\end{fact}
\begin{proof}
   The trace is a conjugacy invariant, so reversibility mod $n$
   implies $\trace (M) \equiv \trace (M^{-1}) \mod n$.  The inversion
   formula for $2\times 2$ matrices yields $\trace (M^{-1}) \equiv
   \frac{\trace (M)}{D}\equiv - \trace (M) \mod n$, and thus $2\ts
   \trace (M) \equiv 0\mod n$.
\end{proof}

\begin{fact}\label{InvolTr0}
   Consider\/ $M\in \Mat(2,\ZZ)$ with\/ $D\equiv -1$ mod\/ $n$.  Then,
   $M$ is an involution mod\/ $n$ if and only if\/ $\trace (M) \equiv
   0$ mod\/ $n$.
\end{fact}
\begin{proof}
Let $M=\left(\begin{smallmatrix} a & b \\ c & d \end{smallmatrix}
\right)$. With $D \equiv -1$, the inversion formula for $M$
shows  that $M\equiv M^{-1}$ is equivalent to $d \equiv - a$.
Thus, $M^2\equiv \one$ if and only if $\trace (M) \equiv 0$.
\end{proof}

The previous two facts imply
\begin{coro}
   Let\/ $M \in \Mat (2,\ZZ)$ be reversible mod\/ $n > 2$ with
   $D\equiv -1$ mod\/ $n$. Then, $M^2 \equiv \one$ mod\/ $n$ for\/ $n$
   odd, and\/ $M^2 \equiv \one$ mod $n/2$ for\/ $n$ even.  \qed
\end{coro}

Let us continue with the general arguments and formulate a necessary
condition for local reversibility.
\begin{lemma}\label{necCond}
   Let\/ $p\not= 2$ be a prime.  If\/ $M \in \Mat(d,\ZZ)$ is
   reversible mod $p^r$, one has $D \equiv \pm 1$ mod\/ $p^r$.  
   If\/ $d=2$, $M$ is reversible mod\/ $p^r$ if and
   only if\/ $D\equiv 1$ or\/ $M^2\equiv\one$ mod\/ $p^r$.

   If\/ $M \in \Mat(d,\ZZ)$ is reversible mod\/ $2^r$, then\/ $D
   \equiv \pm 1$ mod\/ $2^{r-1}$. When\/ $d=2$ and\/ $M$ is reversible
   with\/ $D\equiv -1$ mod\/ $2^{r-1}$, one has\/ $M^2 \equiv \one$
   mod\/ $2^{r-2}$.
\end{lemma}
\begin{proof}
For $p\not= 2$, Fact~\ref{Dsquare} implies $D^2 \equiv 1$ mod
$p^r$. Since $p$ cannot divide both $D-1$ and $D+1$, one has $p^r |
(D-1)$ or $p^r | (D+1)$, which gives the first claim.  When $2^r |
(D-1)(D+1)$, $2$ divides one of the factors and $2^{r-1}$ the other
one, so $D\equiv 1$ or $D\equiv -1$ mod $2^{r-1}$.  If $D\equiv -1$
mod $2^{r-1}$, Fact~\ref{tr0} gives $2\ts\trace (M) \equiv 0$ mod
$2^{r-1}$ and thus $M^2\equiv \one$ mod $2^{r-2}$ by
Fact~\ref{InvolTr0}.
\end{proof}

One immediate consequence for $d=2$ is the following.
\begin{coro}
   If\/ $M\in \GL(2,\ZZ)$ with\/ $D=-1$ is reversible for infinitely
   many\/ $n \in \NN$, one has\/ $M^2 = \one$.  \qed
\end{coro}

\begin{fact}\label{reduction}
   Let\/ $A$ be an integer matrix whose determinant is coprime with\/
   $n\in\NN$.  The reduction of the inverse of $A$ over $\ZZ/n\ZZ$,
   taken modulo $k|n$, is then the inverse of $A$ over $\ZZ/k\ZZ$.  \qed
\end{fact}

\begin{lemma}\label{lem:conjugacy}
  Let\/ $n = p_1^{r_1}\ldots p_s^{r_s}$ be the prime decomposition
  of\/ $n\in\NN$. Then, two matrices $M, M' \in \Mat(d,\ZZ)$ are
  conjugate mod\/ $n$ if and only if they are conjugate mod\/
  $p_i^{r_i}$ for all\/ $1 \le i \le s$.
\end{lemma}
\begin{proof}
$M\sim M'$ mod $n$ means $M' = A M A^{-1}$ for some $A\in \GL (n,\ZZ)$,
which implies conjugacy mod $k$ for all $k|n$.

For the converse, let $A_i \in \GL(d,\ZZ/p_{i}^{r_i}\ZZ)$ denote the
conjugating matrix mod $p_i^{r_i}$. The Chinese remainder theorem,
applied to each component of the matrices $A_i$ and $A_i^{-1}$,
respectively, gives matrices $A$ and $B$ that reduce to $A_i$ and
$A_i^{-1}$ modulo $p_i^{r_i}$, respectively. By construction, $A B
\equiv \one\mod p_i^{r_i}$ for all $i$, hence also $ A B \equiv
\one\mod n $ and thus $B = A^{-1}$ in $\GL(d,\ZZ/n \ZZ)$.
\end{proof}

\begin{prop}
  With\/ $n$ as in Lemma~$\ref{lem:conjugacy}$, a matrix $M \in
  \Mat(d,\ZZ)$ is reversible mod\/ $n$ if and only if\/ $M$ is
  reversible mod\/ $p_i^{r_i}$ for all\/ $1 \le i \le s$.
\end{prop}
\begin{proof}
  The claim is a statement about the conjugacy of $M$ and $M^{-1}$ in
  the group $\GL(d,\ZZ/n\ZZ)$, which is thus a consequence of
  Lemma~\ref{lem:conjugacy}. We just have to add that, by
  Fact~\ref{reduction}, the inverse of $M$ mod $n$ reduces to the
  inverse mod $p_i^{r_i}$, so $M R \equiv R M^{-1}\mod p_i^{r_i}$ for
  all $i$. 
\end{proof}

\begin{coro}
  Consider a matrix $M\in \Mat(2,\ZZ)$ with\/ $D=\det(M)$ and let\/ $n
  = p_1^{r_1} p_2^{r_2}\ldots p_s^{r_s}$. When $n$ is not divisible
  by\/ $4$, $M$ is reversible mod\/ $n$ if and only if, for each\/ $1
  \le i \le s$, $D \equiv 1$ or\/ $M^2\equiv \one$ mod\/ $p_i^{r_i}$.
  When $n= 2_{}^{r_1} p_2^{r_2}\ldots p_s^{r_s}$ with\/ $r^{}_{1} \ge 2$,
  $M$ is reversible mod\/ $n$ if and only if it is reversible mod\/
  $2^{r_1}$ and, for all\/ $i>1$, $D\equiv 1$ or $M^2\equiv\one$ mod\/
  $p_i^{r_i}$. 
\end{coro}

\begin{proof}
According to Lemma~\ref{lem:conjugacy}, the matrix $M$ is reversible
mod $n$ if and only if it is reversible mod $p_{i}^{r_i}$ for all
$1\le i \le s$. By Lemma~\ref{necCond}, this is equivalent with
$D\equiv 1$ or $M^{2}\equiv 1$ mod $p_{i}^{r_i}$ for all $i$ with
$4 \nmid p_{i}^{r_i}$.
\end{proof}

\begin{remark}
To see that reversibility mod $p$ for all primes $p$ which divide $n$ is
not sufficient for reversibility mod $n$, one can consider a locally
reversible matrix $M$ with $\det M \not= 1$: according to
Fact~\ref{Dsquare}, only finitely many $n$ exist such
that $M$ is reversible mod $n$, so for each prime $p$ there must be a
maximum $r$ for which $M$ is reversible mod $p^r$.  Recalling
an example from above, $M=\bigl(\begin{smallmatrix} 0 & -4\\ 1 & 0
\end{smallmatrix}\bigr)$ is reversible mod $3$ but not mod $9$ as can
be verified by explicit calculation.  It is an involution mod $5$,
hence also reversible mod $15$, but not mod $45$.
\exend
\end{remark}

Reversibility can be viewed as a structural property that reflects
additional `regularity' in the dynamics, in the sense that it
typically reduces the spread in the period distribution.  For $2\times
2$-matrices, the normal form approach shows that reversibility implies
the existence of only one non-trivial period length on $L_p$; compare
our comments in Section~\ref{sec:Arnold}.

\subsection{Matrix order and symmetries over $\FF_{p}$}\label{sec:4.5}

Let us now discuss the order of a matrix $M \in \GL(d,\FF_{p})$, with $p$
a prime, in conjunction with the existence of roots of $M$ in that
group. We begin by recalling the following result from \cite[Thm.~2.14,
Cor.~2.15 and Cor.~2.16]{LN}.

\begin{fact}\label{F:extensionIrredPoly}
  If\/ $f$ is an irreducible polynomial of degree\/ $d$ over\/
  $\FF_{p}$, its splitting field is isomorphic with\/
  $\FF_{p^d}$. There, it has the\/ $d$ distinct roots\/
  $\alpha,\alpha^p,\ldots,\alpha^{p^{d-1}}$ that are conjugates and
  share the same order in\/ $(\mathbb{F}_{p^d})^{\times}$.

  In particular, two irreducible polynomials over\/ $\FF_{p}$ of the
  same degree have isomorphic splitting fields.  \qed
\end{fact}

{}From now on, we will identify isomorphic fields with each other.  In
particular, we write $\mathbb{F}_{p^d}$ for the splitting field of an
irreducible polynomial of degree $d$ over $\mathbb{F}_p$.

Next, let $K$ be an arbitrary finite field, consider an irreducible,
monic polynomial $f\in K[x]$ of degree $d$, and let $L$ be the
splitting field of $f$.  When $\lambda^{}_{1},\lambda^{}_{2},\ldots
,\lambda^{}_{d}$ are the roots of $f$ in $L$, one has the well-known
factorisation
\begin{equation} \label{eq:splits}
   f(x)\, =\, \prod_{j=1}^{d} (x-\lambda_j) 
   \, = \, x^d - e^{}_1(\lambda_1,\ldots, \lambda_d) +\ldots + 
   (-1)^d e^{}_d(\lambda_1,\ldots, \lambda_d) \ts ,
\end{equation}
where the $e_{i}$ denote the elementary symmetric polynomials, 
\[
   e^{}_{1} (x^{}_{1},\ldots , x^{}_{d}) 
   = x^{}_{1} + x^{}_{2} + \ldots +
   x^{}_{d} \, , \; \ldots \; , \, 
   e^{}_{d} (x^{}_{1},\ldots , x^{}_{d}) = x^{}_{1}
   \cdot x^{}_{2} \cdot \ldots \cdot x^{}_{d} \ts .
\]
The elementary symmetric polynomials, when evaluated at the roots of
$f$, are fixed under all Galois automorphisms of the field extension
$L/K$, so that the following property is clear.

\begin{fact}\label{f:esp}
  An irreducible, monic polynomial\/ $f\in K[x]$
  satisfies~\eqref{eq:splits} over its splitting field\/ $L$.  In
  particular, the elementary symmetric polynomials\/ $e^{}_{1},
  \ldots, e^{}_{d}$, evaluated at the\/ $d$ roots of\/ $f$ in\/ $L$,
  are elements of\/ $K$.   \qed
\end{fact}

Let $M$ be a $d\times d$ integer matrix with irreducible
characteristic polynomial $\chi^{}_{M}$ over $\FF_{p}$.  Let $\alpha$
be a root of $\chi^{}_{M}$ in $\FF_{p^d}$ and $\lambda$ a generating
element of the unit group $(\FF_{p^d})^{\times}$.  Clearly, there is
an $n\in\NN$ with $\alpha = \lambda^n$.  By Fact
\ref{F:extensionIrredPoly}, one has
$\FF_{p}(\alpha)=\FF_{p^d}=\FF_{p}(\lambda)$, where the degree of the
extension field over $\FF_{p}$ equals $d$.  Consequently, the minimal
polynomial of $\lambda$ over $\FF_{p}$ is an irreducible monic
polynomial of degree $d$ over $\FF_{p}$, and the conjugates of
$\alpha$ are powers of the conjugates of $\lambda$. Let
$\alpha^{}_{1},\ldots,\alpha^{}_{d}$ and $\lambda^{}_{1},\ldots,
\lambda^{}_{d}$ denote the respective collections of conjugates.
Thus, over $\FF_{p^d}$, one has the matrix conjugacy 
\[
   M \,\sim\, \diag(\alpha^{}_{1},\ldots,\alpha^{}_{d}) 
   \, = \, \diag(\lambda^{}_{1},\ldots,\lambda^{}_{d})^n
   \,\sim\, C(f)^n,
\]
with $f(x)\in \FF_{p} [x]$ as in \eqref{eq:splits} and $C(f)$
denoting the companion matrix of $f$.  Here, it was exploited that a
$d\times d$ matrix whose characteristic polynomial $f$ has $d$ distinct
roots is always similar to the companion matrix of $f$. Note
that $C(f) \in \GL(d,\FF_{p})$ by Fact~\ref{f:esp}.

Now, $M$ and $C(f)$ are matrices over $\FF_{p}$ that are conjugate
over $\FF_{p^d}$, so (by a standard result in algebra, see
\cite[Thm.~5.3.15]{AW}) they are also conjugate over $\FF_{p}$, which
means that we have the relation
\begin{equation} \label{eq:root}
    M \, = \, A^{-1} C(f)^n A 
    \, = \, (A^{-1} C(f) A)^n 
    \, =: \, W^n
\end{equation}
with some $A\in\GL(d,\FF_{p})$.  By similarity, $\ord(W) = \ord(C(f)) =
\ord(\diag(\lambda^{}_{1},\ldots,\lambda^{}_{d}))=p^d-1$.  This gives the
following result.

\begin{lemma}\label{lem:root}
  A matrix\/ $M\in \GL(d,\FF_{p})$ with irreducible characteristic
  polynomial has either the maximally possible order\/ $p^d-1$, or
  admits an $n$-th root\/ $W\in\GL(d,\FF_{p})$ as in \eqref{eq:root}.
  Here, $n$ can be chosen as $n = \frac{p^d-1}{\ord(M)}$, so that
  the root has order $p^d-1$.  \qed
\end{lemma}

\begin{fact}\label{f:linComb}
  Let\/ $A$ be a matrix over\/ $\FF_{p}$ with minimal polynomial of
  degree\/ $d$.  Then, the ring
\[
   \FF_{p} [A] \, = \, \{ \xi_1 \one + \ldots + 
   \xi_{d} A^{d-1} \mid  \xi_j\in\FF_{p} \}
\]
has precisely\/ $p^d$ elements, which correspond to the different\/
$d$-tuples\/ $(\xi_1,\ldots,\xi_d)$.
\end{fact}
\begin{proof}
  Two distinct $d$-tuples producing the same matrix would give rise to
  a non-trivial linear combination that vanishes, involving powers of
  $A$ of degree $d-1$ at most, which contradicts the minimal
  polynomial having degree $d$.
\end{proof}

\begin{lemma}\label{lem:root-sym}
  Let\/ $W, M \in \GL (d,\FF_{p})$ satisfy\/ $W^n = M$ and\/
  $\ord (W) = p^{d} -1$.  Then, $\FF_{p} [M] = \FF_{p} [W]$
  and
\[
   \FF_{p} [M]^{\times} \, = \;
   \FF_{p} [M]\setminus \{0\} 
   \, = \, \langle W \rangle
   \, \simeq \, C_{p^{d} - 1} \ts ,
\]
  where\/ $\langle W \rangle$ denotes the cyclic group generated 
  by\/ $W$.
\end{lemma}
\begin{proof}
  Clearly, $\FF_{p} [M]= \FF_{p} [W^n] \subset \FF_{p} [W]$, while
  Fact \ref{f:linComb} implies
  $\left| \FF_{p} [M] \right| = \left| \FF_{p} [W] \right| = p^d$,
  whence we have equality.  Further,
\[
   \langle W \rangle \, \subset \: 
   \FF_{p} [W]^{\times} \, \subset \;
   \FF_{p} [W]\setminus \{ 0 \}
   \, = \, \FF_{p} [M]\setminus\{ 0 \}\ts ,
\]
and again, comparing cardinalities, one finds $\left| \langle W
  \rangle \right| = p^{d}-1 = \left| \FF_{p}[M] \setminus \{ 0 \}
\right|$, from which the claim follows.
\end{proof}

Let us summarise and extend the above arguments as follows.
\begin{coro}
  A\/ $d \!\times\! d$ integer matrix\/ $M$ with irreducible
  characteristic polynomial over the field\/ $\FF_{p}$ has a primitive
  root\/ $W\in\GL(d,\FF_{p})$ with $\ord (W) = p^d-1$. Moreover, one
  then has $\FF_{p}[M]^{\times} = \FF_{p}[M]\setminus\{ 0 \} = \langle
  W \rangle \simeq C_{p^{d}-1}$.  In particular, $\cS (M) \simeq
  C_{p^d-1}$ in this case.

  More generally, we have $\cS (M) = \FF_{p} [M]^{\times}$
  whenever the minimal polynomial has degree $d$. 
\end{coro}

\begin{proof}
  Since we work over the field $\FF_{p}$, the irreducibility of the
  characteristic polynomial of $M$ means that the minimal polynomial
  agrees with the characteristic polynomial and has thus maximal
  degree $d$. This situation is equivalent with $M$ being cyclic
  \cite[Thm.~III.2]{J}. By Thm.~17 of \cite{J} and the Corollary
  following it, we know that any matrix which commutes with $M$ is
  a polynomial in $M$, so that $\cS (M) = \FF_{p} [M]^{\times}$ is
  clear.

  The claim for matrices $M$ with an irreducible characteristic 
  polynomial follows by Lemmas~\ref{lem:root} and \ref{lem:root-sym}.
\end{proof}

When a matrix $M\in \Mat(d,\FF_{p})$ fails to be cyclic, there are
always commuting matrices that are not elements of $\FF_{p} [M]$,
see Thm.~19 of \cite{J} and the following Corollary. In such a case,
$\cS (M)$ is a true group extension of $\FF_{p} [M]^{\times}$.
The situation is thus particularly simple for matrices
$M\in \Mat(2,\FF_{p})$: Either they are of the form $M=a\one$
(then with $\cS (M) = \GL (2,\FF_{p})$), or they are cyclic
(then with $\cS (M) = \FF_{p} [M]^{\times}$).

\section*{Appendix: Two classic examples}

If one reads through the literature, two matrices are omnipresent
as examples, the Arnold and the Fibonacci cat map. Still, several
aspects of them are unclear or conjectural, despite the effort
of many. Let us sum up some aspects, with focus on properties
in line with our above reasoning.

\subsection*{$\mathrm{A}.1.$ Arnold's cat map}\label{sec:Arnold}

Here, we collect some results for the matrix $M_{\rA}= \left(
  \begin{smallmatrix} 2 & 1 \\ 1 & 1 \end{smallmatrix}\right) \in\SL
(2,\ZZ)$ in an informal manner. This case was studied in
\cite{PV,DF,G} and appeared in many other articles as main example.
It was introduced in \cite[Example 1.16]{AA} as a paradigm of
(discrete) hyperbolic dynamics.

The integer matrix $M_{\rA}$ is reversible within the group
$\GL(2,\ZZ)$, with a reversor of order $4$, but none of order $2$.
One has $\cS (M_{\rA}) \simeq C_{2} \times C_{\infty}$, where $C_{2} =
\{ \pm 1 \}$ and the infinite cyclic group is generated by the unique
square root of $M_{\rA}$ in $\GL(2,\ZZ)$ (see below), while
$\cR (M_{\rA}) = \cS (M_{\rA}) \rtimes C_{4}$; see \cite{BRcat}
for more. In particular, $M_{\rA}$ inherits local reversibility
in $\GL (2, \ZZ/p\ts \ZZ)$ for all primes $p$ from its `global'
reversibility within $\GL (2,\ZZ)$.

It was shown in \cite{G} that $M_{\rA}$, except for the trivial fixed
point $0$, has orbits of only one period length on each prime lattice
$L_{p}$. In view of the normal forms, this is clear whenever the
characteristic polynomial is irreducible. However, a matrix of type
III from Table~\ref{tab:finite-field} has reducible characteristic
polynomial and occurs for primes with $\bigl( \frac{5}{p}\bigr) =-1$.
Here, different orbit lengths would still be possible in general, but
reversibility forces the two roots to be multiplicative inverses of
one another and thus to have the same order modulo $p$.

The iteration numbers are $p_{m} = f_{2m}$, where the $f_{k}$ are the
Fibonacci numbers, defined by the recursion $f_{k+1} = f_{k} +
f_{k-1}$ for $k\in\NN$ with initial conditions $f_{0}=0$ and $f_{1} =
1$. Since $\mgcd (M_{\rA}) = 1$, Proposition~\ref{prop:ord-versus-per}
implies
\[
   \ord (M_{\rA}, n) \, = \, \per^{}_{\rA} (n) \, = \,
   \mathrm{period} \bigl\{  (f_{2m})_{m\ge 0} \bmod{n} \bigr\} ,
\]
where the periods for prime powers (with $r\in\NN$) are given by
\[
   \per^{}_{\rA} (2^r) = 3 \cdot 2^{\max \{0, r-2\} }
   \quad \mbox{and} \quad
   \per^{}_{\rA} (5^r) = 10 \cdot 5^{r-1}
\]
together with
\[
   \per^{}_{\rA} (p^{r}) \, = \, p^{r-1} \ts \per^{}_{\rA} (p)
\]
for all remaining plateau-free primes. It has been conjectured that
this covers all primes \cite{W}. No exception is known to date; the
conjecture was tested for all $p < 10^8$ in \cite{ADS}. Note that each
individual prime can be analysed on the basis of
Proposition~\ref{prop:powers-modp}.

The periods mod $p$ are $\per^{}_{\rA} (2) = 3$,
$\per^{}_{\rA} (5) = 10$, together with
\[
    \per^{}_{\rA} (p) \, = \, \frac{p - \bigl(\frac{5}{p}\bigr)}
    {2 \ts m^{}_{p} - \frac{1}{2} 
    \Big( 1- \bigl(\frac{5}{p}\bigr)\Big)} 
\] 
for odd primes $p\ne 5$, where $\bigl(\frac{5}{p}\bigr)$ denotes the
Legendre symbol and $m^{}_{p} \in \NN$ is a characteristic integer
that covers the possible order reduction. It is $1$ in `most' cases
(in the sense of a density definition), but there are infinitely
many cases with $m^{}_{p} > 1$; this integer is tabulated to some 
extent in \cite{W,G}.

Let us write down the generating polynomials for the
distribution of cycles on the lattices $L_{n}$. Once again,
this is only necessary for $n$ a prime power. We use a
formulation with a factorisation that shows the structure
of orbits on $L_{p^{r}} \setminus L_{p^{r-1}}$. In the
notation of \cite{BRW}, one finds $Z_{1} (t) = (1-t)$ and
\[
   Z_{2^{r}} (t) \, = \, (1-t) (1-t^3)
   \prod_{\ell=0}^{r-2}
   \bigl( 1-t^{3\cdot 2^{\ell}} \bigr)^{4\cdot 2^{\ell}}
\]
with $r\ge 1$ for the prime $p=2$, as well as
\[
   Z_{5^{r}} (t) \, = \, (1-t) \prod_{\ell=0}^{r-1}
   \bigl( (1-t^{2\cdot 5^{\ell}}) (1-t^{10\cdot 5^{\ell}})
   \bigr)^{2\cdot 5^{\ell}}
\]
with $r\ge 1$ for $p=5$. As usual, we adopt the convention to treat an
empty product as $1$.  The remaining polynomials read
\[
   Z_{p^{r}} (t) \, = \, (1-t) \prod_{\ell=0}^{r-1}
   \bigl( 1-t^{\ts\per^{}_{\rA} (p) \ts p^{\ell}} 
   \bigr)^{\frac{p^{2} - 1}{\per^{}_{\rA} (p)} \, p^{\ell}} ,
\]
as long as the plateau phenomenon is absent (see above).

\subsection*{$\mathrm{A}.2.$ Fibonacci cat map}
Closely related is the matrix $M_{\rF}= \left(
\begin{smallmatrix} 1 & 1 \\ 1 & 0 \end{smallmatrix}\right)
\in\GL (2,\ZZ)$, which is the unique square root of the Arnold cat map
$M_{\rA}$ in $\GL(2,\ZZ)$. It appears in numerous applications; see
\cite{RB,trace,BHP,David} and references therein for some of them.
Here, the iteration numbers are the Fibonacci numbers themselves, and
the periods are the so-called \emph{Pisano periods}; compare
\cite[A\ts 001175]{online} and references given there, or \cite{W}.

The matrix $M_{\rF}$ is not reversible in $\GL(2,\ZZ)$ (while its
square $M_{\rA}$ is, see above), and has the same symmetry group
as $M_{\rA}$. In fact, $\pm M_{\rF}$ are the only roots of $M_{\rA}$
in $\GL(2,\ZZ)$. This situation implies that the orbit structure
for $M_{\rF}$ must be such that the iteration of its square gives
back the counts we saw in the previous example.

For prime powers $p^r$, with $r\in\NN$, one finds $\per^{}_{\rF}
(5^{r}) = 20 \cdot 5^{r-1}$ together with
\[
    \per^{}_{\rF} (p^{r}) \, = \, p^{r-1} \ts \per^{}_{\rF} (p)
\]
for all remaining primes, with the same proviso as for the Arnold cat
map. The periods $\per^{}_{\rF} (p)$ are given by
$\per^{}_{\rF} (2) = \per^{}_{\rA} (2) = 3$ together with
\[
   \per^{}_{\rF} (p) \, = \, 2 \ts \per^{}_{\rA} (p)
\]
for all odd primes, which is not surprising in view of the relation
between the two matrices $M_{\rF}$ and $M_{\rA}$.

The orbit distribution is more complicated in this case, as
usually orbits of two possible lengths arise in each step.
One finds
\[
    Z_{2^{r}} (t) \, = \, (1-t) 
   \prod_{\ell=0}^{r-1}
   \bigl( 1-t^{3\cdot 2^{\ell}} \bigr)^{2^{\ell}}
\]
and
\[
   Z_{5^{r}} (t) \, = \, (1-t) \prod_{\ell=0}^{r-1}
   \bigl( (1-t^{4\cdot 5^{\ell}}) (1-t^{20\cdot 5^{\ell}})
   \bigr)^{5^{\ell}}
\]
for the primes $2$ and $5$ (with $r\in\NN_{0}$ as before),
as well as
\[
   Z_{p^{r}} (t) \, = \, (1-t) \prod_{\ell=0}^{r-1}
    \bigl( 1-t^{\frac{1}{2}\ts\per^{}_{\rF} (p) \ts p^{\ell}} 
   \bigr)^{2\ts n^{}_{p}}
   \bigl( 1-t^{\ts\per^{}_{\rF} (p) \ts p^{\ell}} 
   \bigr)^{\frac{p^{2} - 1}{\per^{}_{\rF} (p)} \, 
   p^{\ell} - n^{}_{p}}
\]
for all remaining primes that are free of the plateau
phenomenon (which possibly means all, see above). Here,
$n^{}_{p}\in \NN_{0}$ is a characteristic integer which often
takes the values $1$ or $0$, but does not seem to be bounded.

\smallskip
\section*{Acknowledgements}
It is our pleasure to thank A.~Weiss for his cooperation and
R.V.~Moody for helpful discussions. This work was supported by the
Australian Research Council (ARC), via grant DP\ts 0774473, and by 
the German Resarch Council (DFG), within the CRC 701.

\end{document}